\documentclass[12pt]{amsart}
\usepackage{verbatim}
\usepackage{latexsym}
\usepackage{amsmath}
\usepackage{amsfonts}
\usepackage{amssymb}
\usepackage{graphicx}
\DeclareMathOperator{\arcsinh}{arcsinh}

\renewcommand{\mod}[1]{\,{\rm mod}\,#1}
\newcommand{\D}{\mathbb{D}}

\newcommand{\C}{\mathbb{C}}

\newcommand{\Z}{\mathbb{Z}}

\newcommand{\TT}{\mathbb{T}}%Torus
\newcommand{\HH}{\mathbb{H}}
\newtheorem{thm}{Theorem}

\newtheorem{prop}[thm]{Proposition}

\newtheorem{lem}[thm]{Lemma}

\newtheorem{rem}[thm]{Remark}

\newcommand{\m}{\mathrm{m}}

\newcommand{\Li}{\mathrm{Li}} %Polylogarithm
\newcommand{\re}{\mathop{\mathrm{Re}}} % Real part
\newcommand{\im}{\mathop{\mathrm{Im}}} % Imaginary part
\newcommand{\hypgeo}[2]{%
  {\vphantom{F}}_{#1}\kern-\scriptspace F_{#2}%
}

\usepackage{xcolor}

\begin{document}

\title{Evaluations of the areal Mahler measure of multivariable polynomials}

\author{Matilde N. Lalin}
\author{Subham Roy}
\address{Matilde Lal\'in:  D\'epartement de math\'ematiques et de statistique, Universit\'e de Montr\'eal. CP 6128, succ. Centre-ville. Montreal, QC H3C 3J7, Canada}\email{matilde.lalin@umontreal.ca}
\address{Subham Roy:  D\'epartement de math\'ematiques et de statistique, Universit\'e de Montr\'eal. CP 6128, succ. Centre-ville. Montreal, QC H3C 3J7, Canada}\email{subham.roy@umontreal.ca}

\begin{abstract}
 We exhibit some nontrivial evaluations of the areal Mahler measure of multivariable polynomials, defined by Pritsker \cite{Pritsker} by considering the integral over the product of unit disks instead of the unit torus as in the standard case. As in the case of the classical Mahler measure, we find examples yielding special values of $L$-functions.    
\end{abstract}


\maketitle 

\section{Introduction}
The (logarithmic) Mahler measure of a non-zero rational function $P \in \C(x_1,\dots,x_n)^\times$ is a height given by 
\begin{equation*}
 \m(P):=\frac{1}{(2\pi i)^n}\int_{\mathbb{T}^n}\log|P(x_1,\dots, x_n)|\frac{d x_1}{x_1}\cdots \frac{d  x_n}{x_n},
\end{equation*}
where the integration is taken over the unit torus $\mathbb{T}^n=\{(x_1,\dots,x_n)\in \mathbb{C}^n : |x_1|=\cdots=|x_n|=1\}$ with respect to the Haar measure. 

In the case where $P$ is a single variable polynomial, $\m(P)$ can be expressed in terms of the roots of $P$ by means of Jensen's formula.  Namely, if 
$P(x)=a \prod_{j=1}^d (x-\alpha_j) \in \C[x]$, then 
\begin{equation}\label{eq:Jensen}\m(P)=\log |a| + \sum_{j=1}^d \log^+ |\alpha_j|,
\end{equation}
where $\log^+|z|=\log\max \{|z|,1\}$.
In the multivariable case, $\m(P)$ is more mysterious, and many formulas are known where $\m(P)$ is related to special values of the Riemann zeta function, $L$-functions, and other functions of arithmetic significance (for particular $P$'s). For example, Smyth \cite{S1,B1} proved the  formulas
\begin{equation}\label{eq:Smyth}
\m(1+x+y) = \frac{3 \sqrt{3}}{4 \pi} L(\chi_{-3},2),
\end{equation}
\begin{equation}\label{eq:zeta}
\m(1+x+y+z)=\frac{7}{2\pi^2} \zeta(3),
\end{equation}
where  $L(\chi_{-3},s)$ is the Dirichlet $L$-function in the character of conductor 3 and $\zeta(s)$ is the Riemann zeta function. The connection to these special values has been explained in terms of relations to regulators by Deninger \cite{Deninger}, as well as  Boyd \cite{Bo98} and Rodriguez-Villegas \cite{RV} (see also the book of Brunault and Zudilin \cite{BrunaultZudilin} for more details). More specifically, the special values $L(\chi_{-3},2)$ and $\zeta(3)$ in equations \eqref{eq:Smyth} and \eqref{eq:zeta} arise from evaluations of the dilogarithm in the sixth roots of unity and the trilogarithm in $\pm 1$, respectively. 

In this note we consider the (logarithmic) areal Mahler measure defined by Pritsker \cite{Pritsker} for $P\in \C(x_1,\dots,x_n)^\times$ as 
\[\m_\D(P)=\frac{1}{\pi^n} \int_{\D^n} \log |P(x_1,\dots,x_n)| dA(x_1)\dots dA(x_n),\]
where 
\[\mathbb{D}^n =\{(x_1,\dots,x_n)\in \C^n : |x_1|,\dots, |x_n|\leq 1\} \]
is the product of $n$ unit disks, and the measure is the natural measure in the $A^0$ Bergman space.

In \cite[Theorem 1.1]{Pritsker}, Pritsker proves that for $P(x) =a\prod_{j=1}^d (x-\alpha_j) \in \C[x]$, then 
\begin{equation}\label{eq:arealJensen}\m_\D(P)=\m(P)+\frac{1}{2}\sum_{|\alpha_j|<1} \left(|\alpha_j|^2-1\right)=\log |a| + \sum_{j=1}^d \log^+  |\alpha_j|+\frac{1}{2}\sum_{|\alpha_j|<1} \left(|\alpha_j|^2-1\right),\end{equation}
thus giving a counterpart to Jensen's formula \eqref{eq:Jensen}. Pritsker proves various  inequalities comparing $\m_\mathbb{D}(P)$ to $\m(P)$ and the  coefficients of $P$, discusses Kronecker's Lemma (that characterizes the one-variable  polynomials $P \in \Z[x]$ such that $\m_\mathbb{D}(P)=0$), considers approximations, exhibits a counterexample to Lehmer's conjecture in this context, and proves some simple evaluations of multivariable cases. Additional results about the areal Mahler measure can be found in the works of Choi and Samuels \cite{ChoiSamuels} and Flammang \cite{Flammang}.

In this note, we evaluate the areal Mahler measure of some nontrivial multivariable polynomials and rational functions. As mentioned above,  some simple evaluations  are included in \cite{Pritsker}. More precisely, Pritsker shows that $\m_\D(x_1+x_2)=-\frac{1}{4}$ and $\m_\D(1+x_1^{k_1}\cdots x_n^{k_n})=0$ for $k_1,\dots,k_d\geq 0$. We provide many more formulas, most of which involve special values of $L$-functions and other special functions. Our results provide evidence that the areal Mahler measure is also an interesting object deserving of attention. For example, we prove the following result. 

\begin{thm}\label{thm:5} We have
\begin{equation}\label{eq:arealsmyth}
\m_\D(1+x+y)=\frac{3\sqrt{3}}{4\pi}L(\chi_{-3},2)+\frac{1}{6}- \frac{11\sqrt{3}}{16\pi}.\end{equation}
\end{thm}

Comparing this formula with \eqref{eq:Smyth}, we see the same term involving the $L$-function/dilogarithm, and some extra terms.

It is natural to wonder if one can obtain an elegant areal Mahler measure formula for  polynomials of the type $a+bx+cy$, where $a,b,c,$ are fixed coefficients, since such formula does exist in the classical case, as proven by Cassaigne and Maillot \cite{CM}. 
It seems to be quite difficult to obtain such  a formula for the areal Mahler measure in full generality. To illustrate this, we  have the following nontrivial statement. 
\begin{thm}\label{thm:sqrt2}
We have  
\[\m_{\mathbb{D}} (\sqrt{2}+x+y)=\frac{L(\chi_{-4},2)}{\pi} + \mathcal{C}_{\sqrt{2}}+ \frac{3}{8}-\frac{3}{2\pi} ,\] where \[\mathcal{C}_{\sqrt{2}} = \frac{\Gamma\left(\frac{3}{4}\right)^2}{\sqrt{2\pi^3}} \hypgeo{4}{3}\left(\frac{1}{4}, \frac{1}{4}, \frac{3}{4}, \frac{3}{4}; \frac{1}{2}, \frac{5}{4}, \frac{5}{4}; 1\right) - \frac{\Gamma\left(\frac{1}{4}\right)^2}{72\sqrt{2\pi^3}} \hypgeo{4}{3}\left(\frac{3}{4}, \frac{3}{4}, \frac{5}{4}, \frac{5}{4}; \frac{3}{2}, \frac{7}{4}, \frac{7}{4}; 1\right),\]
is expressed in terms of generalized hypergeometric functions
defined as 
\begin{equation*}
\hypgeo{p}{q} (a_1,\ldots,a_p;b_1,\ldots,b_q;z) = \sum_{n=0}^\infty \frac{(a_1)_n\cdots(a_p)_n}{(b_1)_n\cdots(b_q)_n} \, \frac {z^n} {n!},\end{equation*}
 where $(a)_n$ denotes the Pochhammer symbol given by 
 $(a)_0 = 1$, and for $n\geq 1$, 
\[
(a)_n = a(a+1)(a+2) \cdots (a+n-1).\]
\end{thm}
 \begin{rem}
While other formulas in this article are given in terms of polylogarithms, it is unclear 
that the term $\mathcal{C}_{\sqrt{2}}$ can be expressed in this form.     
\end{rem}

We notice that 
 \[\m(\sqrt{2}+x+y)=\frac{L(\chi_{-4},2)}{\pi}+\frac{\log 2}{4}.\]
 This formula can be obtained by specializing the more general expression for $\m(a+bx+cy)$ from \cite{CM}.

 A motivation to study this particular polynomial (with $a=\sqrt{2}$ and $b=c=1$) lies in the fact that it is relatively easy to understand the boundaries of integration upon application of Jensen's formula, due to the particular properties of the constant $\sqrt{2}$. We see nevertheless that the formula for $\m_\D(\sqrt{2}+x+y)$ is quite involved.

We also prove the following statement involving a rational function. 
\begin{thm}\label{thm:8}
We have  \begin{equation}\label{eq:arealboyd}\m_{\mathbb{D}} \left(y+\left(\frac{1-x}{1+x}\right)\right)=\frac{6}{\pi}L\left(\chi_{-4}, 2\right)- \log 2 - \frac{1}{2} - \frac{1}{\pi} .\end{equation}
\end{thm}
The above formula can be compared to the evaluation due to Boyd \cite{Boyd-sharp}
\begin{equation}\label{eq:boyd}\m\left(y+\left(\frac{1-x}{1+x}\right)\right)=\frac{2}{\pi} L(\chi_{-4},2).\end{equation}
In this case, the term $L(\chi_{-4},2)$ involving the Dirichlet $L$-function in the character of conductor 4 comes from evaluating the dilogarithm at $\pm i$. However, unlike the situation of equations \eqref{eq:Smyth} and \eqref{eq:arealsmyth}, the dilogarithmic terms in \eqref{eq:arealboyd} and \eqref{eq:boyd} do not have the same coefficients.

Variations of Mahler measure such as generalized Mahler measure \cite{GonOyanagi}, multiple and higher Mahler measure \cite{KLO}, and zeta Mahler measures \cite{BernsteinGelfand,Akatsuka} can be adapted to the areal Mahler measure setting. For example, we compute the areal Mahler measure of $x+1$ and prove that 
 \begin{equation}\label{eq:arealzeta} Z_\mathbb{D}(s,x+1):=\frac{1}{\pi}\int_{\mathbb{D}} |x+1|^s dA(x)=\exp \left(\sum_{j=2}^\infty \frac{(-1)^{j}}{j} (1-2^{1-j})(\zeta(j)-1)s^j\right),\end{equation}
which can be compared to the classical case \cite{Akatsuka, KLO}
 \[ Z(s,x+1):=\frac{1}{2\pi}\int_{\mathbb{T}} |x+1|^s \frac{dx}{x}=\exp \left(\sum_{j=2}^\infty \frac{(-1)^{j}}{j} (1-2^{1-j})\zeta(j)s^j\right).\]
We see that both formulas are very similar, and indeed,
\[ Z_\D(s,x+1)=\frac{s+1}{(s/2+1)^2}Z(s,x+1).\]
The techniques for proving formulas for the areal Mahler measure and its variants are not unlike the techniques employed in the classical Mahler measure case, and include inventive changes of variables in integrals, as well as connections to polylogarithms and other special functions such as generalized hypergeometric series, and their properties. For the moment we lack the connection to regulators that could potentially allow us to perform these evaluations more systematically. Nevertheless, our results open the door to future considerations of the areal Mahler measure and suggest the search for deeper connections to regulators that could potentially explain such formulas.

This article is organized as follows. In Section \ref{sec:basic} we present the areal Mahler measure of $x+y$ and more generally $x_1\cdots x_m+y_1\cdots y_n$ as a prelude to more involved arguments that follow in subsequent sections. Section \ref{sec:polylogarithms} contains the definition and basic properties of polylogarithms, which are central to most of our computations. We prove the main areal Mahler measure formulas given by 
Theorems \ref{thm:5}, \ref{thm:sqrt2}, and \ref{thm:8}  in Section \ref{sec:arealMM}. Finally, in Section 
\ref{sec:generalized} we extend the definition of generalized, multiple, and higher Mahler measure and the zeta Mahler measure to the areal context, and give examples of evaluations in each of these cases, including formula \eqref{eq:arealzeta}. 

\section*{Acknowledgements} The authors are grateful to the anonymous referee for helpful remarks that greatly improved the exposition of this article. This work was supported by  the Natural Sciences and Engineering Research Council of Canada, Discovery Grant 355412-2022, the Fonds de recherche du Qu\'ebec - Nature et technologies, Projet de recherche en \'equipe 300951,  the Institut des sciences math\'ematiques, and the Centre de recherches math\'ematiques.

\section{A basic result}\label{sec:basic}

The simplest possibly non-trivial polynomial that we can consider in this context is a linear polynomial in one variable. Equation \eqref{eq:arealJensen} gives us 
\begin{equation}\label{eq:linear}
\m_\D(x-\alpha)=\begin{cases}
                  \log^+|\alpha| & |\alpha|\geq 1,\\
                  \frac{|\alpha|^2-1}{2} & |\alpha|\leq 1.
                 \end{cases}\end{equation}
Given the above formula, it is natural to pose the question about the areal Mahler measure of $x+y$. The following result is due to Pritsker, but we reprove it here for completeness. 

\begin{prop}\label{prop:x+y}\cite[Example 5.2]{Pritsker} We have 
\[ \m_\D(x+y)=-\frac{1}{4}.\]
\end{prop}
\begin{proof}We first consider the integral over the variable $y$ by exploiting formula \eqref{eq:linear}. This gives 
\[\m_\D(x+y)=\frac{1}{\pi^2} \int_{\D^2} \log|x+y| dA(y) dA(x) =\frac{1}{2\pi} \int_{\D}  (|x|^2-1) dA(x).\]
Parametrizing $x=\rho e^{i \theta}$ with $0\leq \rho \leq 1$ and $-\pi \leq \theta \leq \pi$, the above integral becomes
\[\m_\D(x+y)=\frac{1}{2\pi}\int_{-\pi}^{\pi}\int_0^1 (\rho^2-1)\rho d \rho d\theta=\int_0^1 (\rho^2-1)\rho d \rho =\left. \left(\frac{\rho^4}{4}-\frac{\rho^2}{2}\right)\right|^1_0=-\frac{1}{4}.\]
\end{proof}
We remark that Proposition \ref{prop:x+y} exhibits a point of difference between the classical case $\m$ and the areal case $\m_\mathbb{D}$. Indeed, it is known that the classical Mahler measure of an homogeneous polynomial is the same as the Mahler measure of any dehomogenization, and in particular, $\m(x+y)=\m(x+1)=0$. However, equation \eqref{eq:linear} shows that $\m_\mathbb{D}(x+1)=0$, while Proposition \ref{prop:x+y} gives $\m_\mathbb{D}(x+y)=-\frac{1}{4}$. 

This discrepancy is even more general. Indeed, while $\m(x_1\cdots x_m+y_1\cdots y_n)=0$, the areal Mahler measure $\m_\D(x_1\cdots x_m+y_1\cdots y_n)$ is a rational number  depending on  $m$ and $n$ as shown in the following result. 
\begin{thm}\label{thm:xxxyyy}
We have for $m\geq 1$, 
\[\m_\D(x_1\cdots x_m+y)=\frac{1}{2^{m+1}}-\frac{1}{2}.\]
For $m,n\geq 2$, 
\begin{align*}
\m_\D(x_1\cdots x_m+y_1\cdots y_n)=&\frac{1}{4}+\binom{m+n-2}{m-1}
 \frac{1}{2^{m+n}}\\&-\frac{1}{2^{m+n}} \sum_{r=0}^{n-1} \binom{m+n-3-r}{m-2} 2^r-\frac{1}{2^{m+n}} \sum_{r=0}^{m-1} \binom{m+n-3-r}{n-2} 2^r\\
&
-\frac{m}{2^{m+n+1}}\sum_{r=0}^{n-1} \binom{m+n-1-r}{m} 2^r-\frac{n}{2^{m+n+1}} \sum_{r=0}^{m-1} \binom{m+n-1-r}{n} 2^r.
\end{align*}
\end{thm}

Before proving Theorem \ref{thm:xxxyyy}, we need an auxiliary result. 
\begin{lem}\label{Chu-VanId*}
For $a, b \in \Z_{\geq 1},$ \[\sum_{r=0}^a\binom{a+b-r}{b} 
 2^r+\sum_{r=0}^b\binom{a+b-r}{a} 2^r=2^{a+b+1}.\]
\end{lem}
\begin{proof}
A version of Chu--Vandermonde identity states that, for $m,n,r \in \Z_{\geq 0}$ such that $m+n \leq r$, we have \begin{equation}\label{Chu-VanId}
\sum_{k = n}^{r-m}\binom{r-k}{m} \binom{k}{n} = \binom{r+1}{m+n+1}.
\end{equation}
(See for example Equation (25) in \cite[1.2.6]{Knuth}.)

Since $2^r = \sum_{p=0}^r \binom{r}{p},$ we obtain \begin{align*}
\sum_{r=0}^a\binom{a+b-r}{b}2^r =& \sum_{r=0}^a \binom{a+b-r}{b}\sum_{p=0}^r \binom{r}{p}\\ =& \sum_{p=0}^a \sum_{r=p}^a \binom{r}{p} \binom{a+b-r}{b} \\ =& \sum_{p=0}^a \binom{a+b+1}{b+p+1},
\end{align*}
where the last equality follows from \eqref{Chu-VanId}. 
Similarly, \[\sum_{r=0}^b\binom{a+b-r}{a} 2^r = \sum_{q=0}^b \binom{a+b+1}{a+q+1} = \sum_{q=0}^b \binom{a+b+1}{b-q}.\] Therefore, we have \begin{align*}
\sum_{r=0}^a\binom{a+b-r}{b} 
 2^r+\sum_{r=0}^b\binom{a+b-r}{a} 2^r =& \sum_{p=0}^a \binom{a+b+1}{b+p+1} + \sum_{q=0}^b \binom{a+b+1}{b-q} \\ =& \sum_{m=0}^{a+b+1} \binom{a+b+1}{m} \\ =& 2^{a+b+1},
 \end{align*} where the penultimate inequality follows from rearranging the terms.
\end{proof}

\begin{proof}[Proof of Theorem \ref{thm:xxxyyy}]
Without loss of generality we can assume that $m\geq 2$ and $n\geq 1$. By symmetry, this only excludes the case $m=n=1$, which was already treated in Proposition \ref{prop:x+y}. 
By equation \eqref{eq:linear} applied to $\m_\D(x)=-\frac{1}{2}$, we have $\m_\D(x_1\cdots x_{m-1})=-\frac{m-1}{2}$. 
By multiplicity, we get
\[\m_\D(x_1\cdots x_m+y_1\cdots y_n)=\m_\D\left(x_m+\frac{y_1\cdots y_n}{x_1\cdots x_{m-1}}\right) - \frac{m-1}{2}.\]
Applying the definition of the areal Mahler measure and  integrating respect to $x_m$ by means of \eqref{eq:linear}, we obtain
\begin{align*}
 &\m_\D(x_1\cdots x_m+y_1\cdots y_n) + \frac{m-1}{2} \\=&\frac{1}{\pi^{m+n}}\int_{\D^{m+n}} \log\left|x_m+\frac{y_1\cdots y_n}{x_1\cdots x_{m-1}}\right| dA(x_1)\dots dA(x_m)dA(y_1)\dots dA(y_n)\\
 =& \frac{1}{2\pi^{m+n-1}}\int_{\D^{m+n-1}\cap\{|y_1\cdots y_n|\leq |x_1\cdots x_{m-1}|\} } \left(\left|\frac{y_1\cdots y_n}{x_1\cdots x_{m-1}}\right|^2-1\right) dA(x_1)\dots dA(x_{m-1})dA(y_1)\dots dA(y_n)\\
 &+ \frac{1}{\pi^{m+n-1}}\int_{\D^{m+n-1}\cap\{|y_1\cdots y_n|\geq |x_1\cdots x_{m-1}|\} } \log\left|\frac{y_1\cdots y_n}{x_1\cdots x_{m-1}}\right| dA(x_1)\dots dA(x_{m-1})dA(y_1)\dots dA(y_n).
\end{align*}
We now consider the change of variables to polar coordinates $x_j=\rho_je^{i\theta_j}$ and $y_k=\sigma_k e^{i \tau_k}$, for $j=1,\dots, m-1$ and $k=1,\dots, n$, where $0\leq \theta_j, \tau_k\leq 2\pi$, and $0\leq \rho_j, \sigma_k\leq 1$. Since the functions under consideration are independent of $\theta_j, \tau_k$, we can directly integrate respect to those variables. We have 
 \begin{align*}
 &\m_\D(x_1\cdots x_m+y_1\cdots y_n) + \frac{m-1}{2} \\ =&\frac{2^{m+n-1}}{2}\int_{0}^1 \cdots \int_0^1 \int_{\sigma_1\cdots\sigma_n\leq \rho_1\cdots \rho_{m-1}}
\left(\left(\frac{\sigma_1\cdots \sigma_n}{\rho_1\cdots \rho_{m-1}}\right)^2-1\right)  \rho_1\cdots \rho_{m-1} \sigma_1\cdots \sigma_n d\rho_1\dots d\rho_{m-1} d\sigma_1\dots d\sigma_n\\
&+2^{m+n-1}\int_{0}^1 \cdots \int_0^1 \int_{\sigma_1\cdots\sigma_n\geq \rho_1\cdots \rho_{m-1}}
\log\left(\frac{\sigma_1\cdots \sigma_n}{\rho_1\cdots \rho_{m-1}}\right)\rho_1\cdots \rho_{m-1} \sigma_1\cdots \sigma_n d\rho_1\dots d\rho_{m-1} d\sigma_1\dots d\sigma_n. 
 \end{align*}
We further consider the change of variables $\alpha_j=\rho_1\cdots \rho_j$, and $\beta_k=\sigma_1\cdots\sigma_k$ for $j=1,\dots, m-1$ and $k=1,\dots,n$. This transformation leads to
 \begin{align*}
 &\m_\D(x_1\cdots x_m+y_1\cdots y_n) + \frac{m-1}{2} \\ =&2^{m+n-2}\int_{\substack{0\leq \alpha_{m-1}\leq \cdots \leq \alpha_1\leq 1\\0\leq \beta_{n}\leq \cdots \leq \beta_1\leq 1\\\beta_n\leq \alpha_{m-1}}} 
\left(\frac{\beta_n^3}{\alpha_{m-1}}-\alpha_{m-1}\beta_n\right)  \frac{d\alpha_1\dots d\alpha_{m-1}}{\alpha_1\cdots \alpha_{m-2}} \frac{d\beta_1\dots d\beta_n}{\beta_1\cdots \beta_{n-1}}\\
&+2^{m+n-1}\int_{\substack{0\leq \alpha_{m-1}\leq \cdots \leq \alpha_1\leq 1\\0\leq \beta_{n}\leq \cdots \leq \beta_1\leq 1\\\beta_n\geq \alpha_{m-1}}} 
(\log\beta_n-\log \alpha_{m-1})
  \alpha_{m-1} \beta_n \frac{d\alpha_1\dots d\alpha_{m-1}}{\alpha_1\cdots \alpha_{m-2}} \frac{d\beta_1\dots d\beta_n}{\beta_1\cdots \beta_{n-1}}.
 \end{align*}
Integrating respect to $\alpha_1,\dots, \alpha_{m-2}$ as well as $\beta_1, \dots, \beta_{n-1}$ leads to
 \begin{align*}
&2^{m+n-2}\int_{0\leq \beta_n\leq \alpha_{m-1}\leq 1}
\left(\frac{\beta_n^3}{\alpha_{m-1}}-\alpha_{m-1}\beta_n\right)\frac{(-1)^{m-2}}{(m-2)!} \log^{m-2} \alpha_{m-1} 
\frac{(-1)^{n-1}}{(n-1)!} \log^{n-1} \beta_{n}
d\alpha_{m-1} d\beta_n\\
&+2^{m+n-1}\int_{0\leq  \alpha_{m-1}\leq \beta_n\leq 1}
(\log\beta_n-\log \alpha_{m-1})
  \alpha_{m-1} \beta_n \frac{(-1)^{m-2}}{(m-2)!} \log^{m-2} \alpha_{m-1} 
\frac{(-1)^{n-1}}{(n-1)!} \log^{n-1} \beta_{n}
d\alpha_{m-1} d\beta_n\\
=&\frac{(-1)^{m+n-1}2^{m+n-2}}{(m-2)!(n-1)!}\int_{0\leq \beta\leq \alpha\leq 1}
\left(\frac{\beta^3}{\alpha}-\alpha\beta\right)\log^{m-2} \alpha 
 \log^{n-1} \beta
d\alpha d\beta\\
&+\frac{(-1)^{m+n-1}2^{m+n-1}}{(m-2)!(n-1)!}\int_{0\leq \alpha \leq  \beta\leq 1}
(\log\beta-\log \alpha)
  \alpha \beta \log^{m-2} \alpha 
 \log^{n-1} \beta
d\alpha d\beta.
 \end{align*}

The above integral can be decomposed into a sum of similar terms, which can be evaluated from the following general formula (see formula 2.722 in \cite{Gradshteyn-Ryzhik}):
\begin{equation}\label{eq:xjlogk}\int x^j \log^k x dx = x^{j+1}\sum_{r=0}^k \frac{(-1)^rr!}{(j+1)^{r+1}} \binom{k}{r}\log^{k-r} x +C.
\end{equation}

Formula \eqref{eq:xjlogk} allows us to compute
\begin{align}\label{eq:ab}\int_{0\leq \beta \leq \alpha \leq 1}\alpha \beta \log^a\alpha \log^b\beta =&\frac{(-1)^{a+b}a!b!}{2^{2a+2b+3}} \sum_{r=0}^b \binom{a+b-r}{a} 2^r.
\end{align}
We also have 
\begin{align}\label{eq:rational}
    \int_{0\leq \beta\leq \alpha\leq 1}
\frac{\beta^3}{\alpha}\log^{m-2} \alpha 
 \log^{n-1} \beta
d\alpha d\beta=&-\frac{1}{m-1} \int_0^1 \beta^3\log^{m+n-2} \beta  d\beta\nonumber \\=& \frac{(-1)^{m+n-1}(m+n-2)!}{(m-1)4^{m+n-1}}.
\end{align}

Combining \eqref{eq:ab}, \eqref{eq:rational}, as well as  Lemma \ref{Chu-VanId*},
we obtain 
 \begin{align}\label{eq:finalmn}
 &\m_\D(x_1\cdots x_m+y_1\cdots y_n) \nonumber\\ 
 =& -\frac{m-1}{2}+ \binom{m+n-2}{m-1} \frac{1}{2^{m+n}}-\frac{1}{2^{m+n-1}} \sum_{r=0}^{n-1} \binom{m+n-3-r}{m-2} 2^r\nonumber\\&
 -\frac{n}{2^{m+n}} \sum_{r=0}^{m-2} \binom{m+n-2-r}{n} 2^r+\frac{m-1}{2^{m+n}} \sum_{r=0}^{m-1} \binom{m+n-2-r}{n-1} 2^r\nonumber\\
=&
 \binom{m+n-2}{m-1}
 \frac{1}{2^{m+n}}-\frac{1}{2^{m+n-1}} \sum_{r=0}^{n-1} \binom{m+n-3-r}{m-2} 2^r\nonumber\\
&-\frac{n}{2^{m+n}} \sum_{r=0}^{m-2} \binom{m+n-2-r}{n} 2^r
-\frac{m-1}{2^{m+n}}\sum_{r=0}^{n-1} \binom{m+n-2-r}{m-1} 2^r.
\end{align}
Specializing the above 
 for $n=1$ and $m>1$,  we obtain 
\[\m_\D(x_1\cdots x_m+y)=\frac{1}{2^{m+1}}-\frac{1}{2}.\]
Moreover, by comparing with  Proposition \ref{prop:x+y}, this formula is also true for $m=1$. 

When $n>1$, expression \eqref{eq:finalmn} can be made symmetric by exchanging $m$ and $n$ and taking the average. Applying  Lemma \ref{Chu-VanId*} again \eqref{eq:ab}, this gives the final expression in the case where both $m,n>1$:
\begin{align*}
\m_\D(x_1\cdots x_m+y_1\cdots y_n)=&\frac{1}{4}+\binom{m+n-2}{m-1}
 \frac{1}{2^{m+n}}\\&-\frac{1}{2^{m+n}} \sum_{r=0}^{n-1} \binom{m+n-3-r}{m-2} 2^r-\frac{1}{2^{m+n}} \sum_{r=0}^{m-1} \binom{m+n-3-r}{n-2} 2^r\\
&
-\frac{m}{2^{m+n+1}}\sum_{r=0}^{n-1} \binom{m+n-1-r}{m} 2^r-\frac{n}{2^{m+n+1}} \sum_{r=0}^{m-1} \binom{m+n-1-r}{n} 2^r.
\end{align*}
\end{proof}

\section{Some preliminary results on dilogarithms and polylogarithms} \label{sec:polylogarithms}
We recall in this section some basic properties of polylogarithms. For 
$n_1,\dots, n_m$ positive integers, one can define the multiple polylogarithm as the complex function given by the series
\begin{equation*}
 \Li_{n_1,\dots,n_m}(z_1,\dots, z_m):=\sum_{0<k_1<\cdots<k_m}\frac{z_1^{k_1}\cdots z_m^{k_m}}{k_1^{n_1}\cdots k_m^{n_m}},
\end{equation*}
which is absolutely convergent for $|z_i|\leq 1$ for $i=1,\dots, m-1$ and $|z_m|\leq 1$ (respectively $|z_m|<1$) if
$n_m>1$ (resp. $n_m=1$). This series can also be extended to a multi-valued meromorphic function on $\C^m$. The number $m$ is called the length  of the polylogarithm, and the number $n_1+\cdots+n_m$ is called the weight. 

The case  $m=1$ and $n_1=1$ leads to the Taylor series of the logarithm at 1,
\[\Li_1(z)=\sum_{k=1}^\infty \frac{z^k}{k}=-\log (1-z).\]
When $m=1$ and $n_1>1$, the evaluation at $z=1$  yields
the Riemann zeta function
\[\Li_{n}(1)=\sum_{j=1}^\infty\frac{1}{j^n}=\zeta(n).\]

Some  combinations of length 2 polylogarithms  can be written in terms of length 1 polylogarithms. This will be useful for us in Theorem \ref{thm:multipleCayley}, which evaluates the areal higher Mahler measure of $\frac{1-x}{1+x}$.
To achieve this simplification we will use some results due to Nakamura \cite{Nakamura} and Panzer \cite{Panzer}. Here we state the formulation of \cite{LalinLechasseur2}.

\begin{thm}\label{thm:LL}\cite[Theorem 3]{LalinLechasseur2} Let $r,s$ be positive integers,  $k=r+s$, and let $u, v$ be complex numbers such that $|u|=|v|=1$. In addition, we assume that $v \not =1$ if $s=1$. 
Let 
\[{\re}_k=\begin{cases} \re & k \mbox{ odd,} \\ i\im & k \mbox{ even.}\end{cases}\]
Then,
\begin{align*}
 2 {\re}_k (\Li_{r,s}(u,v)) =&(-1)^k\Li_k(\overline{uv})+(-1)^{k+1} \Li_r(\overline{u})\Li_s(\overline{v})+ (-1)^{r-1}\Li_r(\overline{u})\Li_s(v)
\\
&+(-1)^{r-1}\left(\binom{k-1}{r-1}\Li_k(\overline{u})+\binom{k-1}{s-1}\Li_k(v)\right)\\
&+ \sum_{m=1}^{k-1}\left( \binom{m-1}{r-1}\Li_m(\overline{u})+ \binom{m-1}{s-1}(-1)^{k+m}\Li_m(v)\right)\\
&\times((-1)^{r} \Li_{k-m}(uv)+(-1)^{s+m}\Li_{k-m}(\overline{uv})).
\end{align*}
\end{thm}

We close this section by considering a modification of the length one dilogarithm. The Bloch--Wigner dilogarithm is given by 
\begin{equation*}
D(z)  =\im(\Li_2(z) -  \log |z| \Li_1(z))  =  \im(\Li_2(z) ) + \log |z| \arg(1-z).
\end{equation*}
It is a continuous function in $\mathbb{P}^1(\C)$, which is real analytic in $\C\setminus \{0,1\}$. It also satisfies 
\begin{equation} \label{eq:intdilog}
-2 \int_0^\theta \log|2 \sin t| d t =  D(e^{2i\theta}).\end{equation}
The following special values will be useful for us
\begin{equation}\label{eq:L3}D(e^{i \pi/3})=\frac{3\sqrt{3}}{4}L(\chi_{-3},2), \end{equation}
 and 
\begin{equation}\label{eq:L4}
D(i)=L(\chi_{-4},2),
\end{equation}
where $L(\chi_{-3},s)$ and $L(\chi_{-4},s)$ 
 denote the Dirichlet $L$-functions on the characters $\chi_{-3}=\left(\frac{-3}{\cdot}\right)$ and $\chi_{-4}=\left(\frac{-4}{\cdot}\right)$ of conductor 3 and 4 respectively. 

\section{Evaluations of the areal Mahler measure} \label{sec:arealMM}

In this section we consider evaluations of the areal Mahler measures of some particular polynomials and  rational functions. 

\subsection{The areal Mahler measure of $1+x+y$}

We now consider Smyth's polynomial $1+x+y$ and prove Theorem \ref{thm:5}.  
\begin{proof}[Proof of Theorem \ref{thm:5}]
By definition and application of \eqref{eq:linear}, we have that
\begin{equation}\label{eq:1+x+y} \begin{split}
\m_\D(1+x+y)=&\frac{1}{\pi^2} \int_{\D^2} \log|1+x+y| dA(y) dA(x) \\=&\frac{1}{2\pi} \int_{\D\cap \{|1+x|\leq 1\}}  (|1+x|^2-1) dA(x) +\frac{1}{\pi} \int_{\D\cap \{|1+x|\geq  1\}} \log|1+x| dA(x).
\end{split}
\end{equation}
We treat the first integral above. Write $x=\rho e^{i \theta}$ with $0\leq \rho \leq 1$ and $-\pi \leq \theta \leq \pi$. We have that 
$|1+x|^2=|1+\rho e^{i \theta}|^2=\rho^2+2\rho \cos \theta +1$, and $|1+x|\leq 1$ if and only if $0\leq \rho \leq -2\cos \theta$ (provided that $\cos \theta\leq 0$). Therefore, when $\frac{2\pi}{3}\leq |\theta| \leq \pi$, we need to integrate in  $0\leq \rho \leq  1$, while for $\frac{\pi}{2}\leq |\theta| \leq \frac{2\pi}{3}$, we need to integrate in  $0\leq \rho \leq -2\cos \theta$. Separating these two cases, we obtain,
\begin{align} \label{eq:first-1+x+y}
 \frac{1}{2\pi} \int_{\D\cap \{|1+x|\leq 1\}}  (|1+x|^2-1) dA(x)=& \frac{1}{\pi} \int_{\frac{2\pi}{3}}^\pi\int_0^1  (\rho^2+2\rho \cos \theta) \rho d\rho d\theta+\frac{1}{\pi}\int_\frac{\pi}{2}^{\frac{2\pi}{3}}\int_0^{-2\cos \theta}(\rho^2+2\rho \cos \theta) \rho d\rho d\theta \nonumber\\
 =&  \frac{1}{\pi} \int_{\frac{2\pi}{3}}^\pi \left(\frac{1}{4} +\frac{2}{3} \cos \theta\right) d\theta+\frac{1}{\pi}\int_\frac{\pi}{2}^{\frac{2\pi}{3}}\left(-\frac{4}{3} \cos^4 \theta\right) d\theta.
\end{align}
Notice that \[ \int_{\frac{2\pi}{3}}^\pi \cos \theta d\theta=-\sin\left(\frac{2\pi}{3}\right)=-\frac{\sqrt{3}}{2},\]
and
\begin{align*}\int_\frac{\pi}{2}^{\frac{2\pi}{3}}\cos^4 \theta d\theta=&\int_\frac{\pi}{2}^{\frac{2\pi}{3}} \left(\frac{1+\cos(2\theta)}{2} \right)^2 d\theta=\int_\frac{\pi}{2}^{\frac{2\pi}{3}} \left(\frac{1}{4}+\frac{\cos(2\theta)}{2} +\frac{\cos^2(2\theta)}{4} \right) d\theta\\
=&\int_\frac{\pi}{2}^{\frac{2\pi}{3}} \left(\frac{1}{4}+\frac{\cos(2\theta)}{2} +\frac{1+\cos(4\theta)}{8} \right) d\theta=\int_\frac{\pi}{2}^{\frac{2\pi}{3}} \left(\frac{3}{8}+\frac{\cos(2\theta)}{2} +\frac{\cos(4\theta)}{8} \right) d\theta\\
=&\frac{\pi}{16}+\frac{1}{4}\sin\left(\frac{4\pi}{3}\right) +\frac{1}{32}\sin\left(\frac{8\pi}{3}\right)
=\frac{\pi}{16}-\frac{7\sqrt{3}}{64}.
\end{align*}
Thus, equation \eqref{eq:first-1+x+y} equals  
 \begin{align} \label{eq:firstterm-1+x+y}
\frac{1}{12}-\frac{1}{\sqrt{3}\pi }-\frac{1}{12}+\frac{7}{16\sqrt{3}\pi}
 =& -\frac{3\sqrt{3}}{16\pi}.
\end{align}

We now consider the second integral in \eqref{eq:1+x+y}.
We make the change of variables $y=1+x$ and set $y=\rho e^{i \theta}$ with $1\leq  \rho $ and $-\pi\leq \theta \leq \pi$. We have that $|y-1|^2=|\rho e^{i \theta}-1|^2=\rho^2-2\rho \cos \theta +1$, and $|y-1|\leq 1$ if and only if $0\leq \rho \leq 2\cos \theta $ (provided that $2\cos \theta \geq 0$). Putting these conditions together, we integrate when $0\leq |\theta|\leq \frac{\pi}{3}$ and $1\leq \rho \leq 2\cos \theta $. This leads to 
\begin{equation}\label{eq:second-1+x+y}
 \begin{split}
 \frac{1}{\pi} \int_{\D\cap \{|1+x|\geq  1\}} \log|1+x| dA(x)=&\frac{2}{\pi} \int_{0}^\frac{\pi}{3}\int_1^{2\cos\theta}  (\log\rho)\rho d\rho d\theta=\frac{2}{\pi} \int_{0}^\frac{\pi}{3}
 \left. \frac{1}{2}\left(\rho^2 \log \rho-\frac{\rho^2}{2} \right) \right|_1^{2\cos\theta}   d\theta\\
 =&\frac{1}{\pi}\int_{0}^\frac{\pi}{3}\left(4\cos^2\theta \log (2\cos\theta)-2\cos^2\theta +\frac{1}{2}\right) d \theta.
\end{split}
\end{equation}
Note that 
\begin{equation}\label{eq:cos2} \int_0^{\frac{\pi}{3}} \cos^2 \theta d\theta=\int_0^{\frac{\pi}{3}}  \left(\frac{1+\cos(2\theta)}{2} \right)d\theta=\frac{\pi}{6}+\frac{1}{4}\sin\left(\frac{2\pi}{3}\right)=\frac{\pi}{6}+\frac{\sqrt{3}}{8}.\end{equation}
It remains to compute the integral of $\cos^2\theta \log (2\cos\theta)$.
We start by making the change of variables $\tau=\frac{\pi}{2}-\theta$ and use \eqref{eq:intdilog} to obtain
\begin{align}\label{eq:cos}
 \int_{0}^\frac{\pi}{3}\cos^2\theta \log (2\cos\theta)d\theta=& \int_{\frac{\pi}{6}}^\frac{\pi}{2} \sin^2 \tau \log(2 \sin \tau) d\tau \nonumber\\
  =&\left. -\frac{1}{2}\sin^2 \tau D\left(e^{2i\tau}\right) \right|_\frac{\pi}{6}^\frac{\pi}{2} +\frac{1}{2} \int_\frac{\pi}{6}^\frac{\pi}{2} 2\sin \tau \cos \tau D\left(e^{2i\tau}\right) d\tau\nonumber\\
=& \frac{1}{8} D(e^{i \pi/3})+\frac{1}{2} \int_\frac{\pi}{6}^\frac{\pi}{2}\sin(2\tau) \sum_{n=1}^\infty \frac{\sin(2n\tau)}{n^2} d\tau\nonumber \\
 =&\frac{1}{8} D(e^{i \pi/3})+\frac{1}{4} \int_\frac{\pi}{3}^\pi \sin(t) \sum_{n=1}^\infty \frac{\sin(nt)}{n^2} dt,
  \end{align}
where we have set $t=2\tau$. 

  Note that
\begin{equation}\label{eq:sin2}\int_\frac{\pi}{3}^\pi\sin^2(t) dt=\frac{\pi}{3}+\frac{\sqrt{3}}{8},\end{equation}
and, for $n\not = 1$,
  \begin{align*}
  \int_\frac{\pi}{3}^\pi\sin(t) \sin(nt) dt =& \frac{1}{2} \int_\frac{\pi}{3}^\pi \left(\cos((n-1)t)-\cos((n+1)t)\right) dt\\
  =& \frac{1}{2}\left.\left(\frac{\sin((n-1)t)}{n-1}-\frac{\sin((n+1)t)}{n+1}\right) \right |_\frac{\pi}{3}^\pi\\
=&-\frac{1}{2} \left(\frac{e^{i(n-1)\pi/3}-e^{-i(n-1)\pi/3}}{2(n-1)i} - \frac{e^{i(n+1)\pi/3}-e^{-i(n+1)\pi/3}}{2(n+1)i} \right).
  \end{align*}
Incorporating the sum for $n\geq 2$ gives
\begin{align*}
 &-\frac{1}{2}   \sum_{n=2}^\infty \left(\frac{e^{i(n-1)\pi/3}-e^{-i(n-1)\pi/3}}{2(n-1)i} - \frac{e^{i(n+1)\pi/3}-e^{-i(n+1)\pi/3}}{2(n+1)i} \right)\frac{1}{n^2}\\
 =&-\frac{1}{2}   \sum_{n=2}^\infty \left(\frac{e^{i(n-1)\pi/3}-e^{-i(n-1)\pi/3}}{2i}\left(\frac{1}{n-1}-\frac{1}{n}-\frac{1}{n^2}\right)
  - \frac{e^{i(n+1)\pi/3}-e^{-i(n+1)\pi/3}}{2i} \left(\frac{1}{n^2}-\frac{1}{n}+\frac{1}{n+1}\right)\right)\\
  =&-\frac{1}{2}\im(\Li_1(e^{i\pi/3}))+\frac{1}{2}\im \left(e^{-i\pi/3}\left(\Li_1(e^{i\pi/3})-e^{i\pi/3}\right)+e^{-i\pi/3}\left(\Li_2(e^{i\pi/3})-e^{i\pi/3}\right)\right)\\
  &+\frac{1}{2}\im \left(e^{i\pi/3}\left(\Li_2(e^{i\pi/3})-e^{i\pi/3}\right)-e^{i\pi/3}\left(\Li_1(e^{i\pi/3})-e^{i\pi/3}\right)\right)+\frac{1}{2}\im\left(\Li_1(e^{i\pi/3})-e^{i\pi/3}-\frac{e^{i2\pi/3}}{2}\right)\\
   =&\frac{1}{2}\im \left(-\sqrt{3}i\Li_1(e^{i\pi/3})+\Li_2(e^{i\pi/3})\right)-\frac{3\sqrt{3}}{8}\\
   =&\frac{1}{2}D(e^{i\pi/3})-\frac{3\sqrt{3}}{8}.
 \end{align*}
By combining this with \eqref{eq:sin2}, and incorporating it in \eqref{eq:cos}, we obtain 
  \[\int_{0}^\frac{\pi}{3}\cos^2\theta \log (2\cos\theta)d\theta=\frac{1}{4} D(e^{i \pi/3})+\frac{\pi}{12}-\frac{\sqrt{3}}{16}.\]
Applying this, as well as \eqref{eq:cos2} and \eqref{eq:L3}, we obtain that 
\eqref{eq:second-1+x+y} equals
\[\frac{3\sqrt{3}}{4\pi}L(\chi_{-3},2)
+\frac{1}{6}-\frac{\sqrt{3}}{2\pi}.\]
Combining the above with \eqref{eq:firstterm-1+x+y} yields the desired result. 

\end{proof}

 \subsection{The areal Mahler measure of $\sqrt{2}+x+y$} 

 We proceed to consider the polynomial $\sqrt{2}+x+y$ and prove Theorem \ref{thm:sqrt2}.

\begin{proof}[Proof of Theorem \ref{thm:sqrt2}]
By definition, we have 
\begin{align}\label{eq:sqrt2}\m_{\mathbb{D}} (\sqrt{2}+x+y)=&\frac{1}{\pi^2}\int_{\mathbb{D}^2}\log|\sqrt{2}+x+y|dA(y) dA(x)\nonumber\\
=&
\frac{1}{2\pi} \int_{\mathbb{D}\cap\{ |\sqrt{2}+x|< 1\}} (|\sqrt{2}+x|^2-1) dA(x)+\frac{1}{\pi}\int_{\mathbb{D}\cap \{|\sqrt{2}+x|\geq 1\}} \log|\sqrt{2}+x| dA(x).
 \end{align}
  We treat the first integral above. Write $x=\rho e^{i \theta}$ with $0\leq \rho \leq 1$ and $0 \leq \theta \leq 2\pi$. We have that 
$|\sqrt{2}+x|^2=|\sqrt{2}+\rho e^{i \theta}|^2=\rho^2+2\sqrt{2}\rho \cos \theta +2$, and $|\sqrt{2}+x|\leq  1$ if and only if $\sin^2\theta \leq \frac{1}{2}$ and  $\max\{0, -\sqrt{2}\cos \theta-\sqrt{1-2\sin^2\theta}\} \leq \rho \leq \min \{ 1, -\sqrt{2}\cos \theta+\sqrt{1-2\sin^2\theta}\}$. Since \[\left|\sqrt{2}\cos \theta\right| \geq \left|\sqrt{1 - 2\sin^2 \theta}\right|, \quad \mbox{and} \quad \cos \theta \in \left[-1, -\frac{1}{\sqrt{2}}\right) \ \mbox{when} \ \left|\pi - \theta\right| \leq \frac{\pi}{4},\] we need to integrate the first integral in  $|\pi - \theta|\leq \frac{\pi}{4}$ and $-\sqrt{2}\cos \theta-\sqrt{1-2\sin^2\theta}\leq \rho \leq 1$. Thus, we have 
\begin{align}\label{eq:first-sqrt2}
& \frac{1}{2\pi} \int_{\mathbb{D}\cap \{|\sqrt{2}+x|< 1\}} (|\sqrt{2}+x|^2-1) dA(x)\nonumber \\&= \frac{1}{2\pi}\int_{\frac{3\pi}{4}}^{\frac{5\pi}{4}} \int_{-\sqrt{2}\cos \theta-\sqrt{1-2\sin^2\theta}}^1  (\rho^2+2\sqrt{2}\rho \cos \theta +1) \rho d\rho d\theta\nonumber\\
 &= \frac{1}{2\pi}\int_{\frac{3\pi}{4}}^{\frac{5\pi}{4}} \left( \frac{8}{3}\cos^4 \theta - 2\cos^2 \theta +\frac{2\sqrt{2}}{3}\cos \theta+ 1+\frac{2\sqrt{2}}{3}\cos \theta (1-2\sin^2\theta)^{3/2}  \right) d\theta.
 \end{align}

We remark that
\begin{align*}
  \int_{\frac{3\pi}{4}}^{\frac{5\pi}{4}} \cos \theta (1-2\sin^2\theta)^{3/2} d\theta =&\left.  \frac{1}{16} \left(3 \sqrt{2} \arcsin(\sqrt{2} \sin\theta ) + 2 (2 \sin\theta + \sin(3 \theta)) \sqrt{\cos(2 \theta)}\right)\right|_{\frac{3\pi}{4}}^{\frac{5\pi}{4}}\\
  =& -\frac{3\sqrt{2}\pi}{16}.
\end{align*}
By proceeding as in the evaluation of \eqref{eq:first-1+x+y}, we have that \eqref{eq:first-sqrt2} becomes
\[\frac{1}{4}-\frac{1}{2\pi}+  \frac{\sqrt{2}}{3 \pi}\left(-\frac{3\sqrt{2}\pi}{16}\right)\\
 =\frac{1}{8}-\frac{1}{2\pi}.\]

For the second integral in \eqref{eq:sqrt2}, we write $y=\sqrt{2}+x$ and $y=\rho e^{i\theta}$ with $1\leq \rho$. We have 
$|y-\sqrt{2}|^2=|\rho e^{i\theta} -\sqrt{2}|^2=\rho^2-2\sqrt{2}\rho \cos \theta+2$ and $|y-\sqrt{2}|\leq 1$ iff $-\frac{\pi}{4}\leq \theta \leq \frac{\pi}{4}$ and $1 \leq \rho \leq \sqrt{2}\cos \theta +\sqrt{1-2\sin^2\theta}$. Therefore, \begin{align}\label{eq:mmsecond}
 \frac{1}{\pi}\int_{\mathbb{D}\cap \{|\sqrt{2}+x|\geq 1\}} \log|\sqrt{2}+x| dx=&\frac{2}{\pi}\int_{0}^{\frac{\pi}{4}} \int_{1}^{\sqrt{2}\cos \theta+\sqrt{1-2\sin^2\theta}} (\log \rho)\rho d\rho d\theta\nonumber \\=&\frac{2}{\pi} \int_{0}^\frac{\pi}{4}
 \left. \frac{1}{2}\left(\rho^2 \log \rho-\frac{\rho^2}{2} \right) \right|_{1}^{\sqrt{2}\cos \theta+\sqrt{1-2\sin^2\theta}}   d\theta\nonumber\\
 =&\frac{1}{\pi}\int_{0}^\frac{\pi}{4}\left(\left( \sqrt{2}\cos \theta+\sqrt{1-2\sin^2\theta}\right)^2 \log\left (\sqrt{2}\cos \theta+\sqrt{1-2\sin^2\theta}\right)\right.\nonumber \\&\left.-\frac{\left(\sqrt{2}\cos \theta+\sqrt{1-2\sin^2\theta}\right)^2}{2}  +\frac{1}{2}\right) d \theta \nonumber \\ =& \frac{1}{8} + \frac{1}{\pi}\int_{0}^\frac{\pi}{4} \left( \sqrt{2}\cos \theta+\sqrt{1-2\sin^2\theta}\right)^2 \log\left (\sqrt{2}\cos \theta+\sqrt{1-2\sin^2\theta}\right) d\theta \nonumber \\ &- \frac{1}{2\pi}\int_{0}^\frac{\pi}{4} \left( \sqrt{2}\cos \theta+\sqrt{1-2\sin^2\theta}\right)^2 d\theta. 
\end{align}

Substituting $x=\sqrt{2}\cos \theta+\sqrt{1-2\sin^2\theta},$ we have $x^{-1}=\sqrt{2}\cos \theta-\sqrt{1-2\sin^2\theta}$, and 
\[d\theta= -\frac{x-x^{-1}}{x\sqrt{4-(x-x^{-1})^2}} dx.\]
This gives 
  \begin{equation}\label{int>1sqrt2pt1}
      \int_{0}^\frac{\pi}{4}\left( \sqrt{2}\cos \theta+\sqrt{1-2\sin^2\theta}\right)^2 \log\left (\sqrt{2}\cos \theta+\sqrt{1-2\sin^2\theta}\right) d\theta =\int_{1}^{\sqrt{2}+1} \frac{x(x-x^{-1})\log x}{\sqrt{4-(x-x^{-1})^2}} dx,
  \end{equation} and \begin{equation}\label{int>1sqrt2pt2}
      \int_{0}^\frac{\pi}{4}\left( \sqrt{2}\cos \theta+\sqrt{1-2\sin^2\theta}\right)^2 d \theta = \int_{1}^{\sqrt{2}+1}\frac{x\left(x-x^{-1}\right)}{\sqrt{4-(x-x^{-1})^2}} dx.
  \end{equation}

Applying integration by parts to  \eqref{int>1sqrt2pt1} gives \begin{align}
&\int_{1}^{\sqrt{2}+1}  \frac{x(x - x^{-1})\log x}{\sqrt{4 - (x - x^{-1})^2}} dx \nonumber \\=& \left[\log x \int^x \frac{u(u - u^{-1})}{\sqrt{4 - (u - u^{-1})^2}}du  - \int \frac{1}{x} \left(\int^x \frac{u(u - u^{-1})}{\sqrt{4 - (u - u^{-1})^2}} du \right) dx  \right]_{1}^{\sqrt{2} + 1}. \label{intbyprt}
\end{align}

We now apply the change of variables $z= u - u^{-1}$ and find that \begin{align}\label{eq:u}
\int \frac{u(u - u^{-1})}{\sqrt{4 - (u - u^{-1})^2}}du =& \frac{1}{4} \int \frac{z\left(z + \sqrt{z^2 + 4}\right)^2}{\sqrt{4 - z^2}\sqrt{4+z^2}} dz \nonumber \\ =& \frac{1}{4} \left( 2\int \frac{z^3 dz}{\sqrt{16 - z^4}} + 4\frac{z dz}{\sqrt{16 - z^4}} + 2\int \frac{z^2 dz}{\sqrt{4 - z^2}}\right) \nonumber \\ =& \frac{1}{4} \left(-\sqrt{16 - z^4} + 2 \arcsin\left(\frac{z^2}{4}\right) - z\sqrt{4-z^2} + 4\arcsin\left(\frac{z}{2}\right)\right),
\end{align}
where the domain under consideration for $z$ is $0 \leq z \leq 2.$ Therefore, the first integral in \eqref{intbyprt} evaluates to \[\frac{1}{4}\left.\log\left(\frac{z + \sqrt{z^2 + 4}}{2}\right) \left(-\sqrt{16 - z^4} + 2 \arcsin\left(\frac{z^2}{4}\right) - z\sqrt{4-z^2} + 4\arcsin\left(\frac{z}{2}\right)\right)\right|_{0}^2 = \frac{3\pi}{4}\log (\sqrt{2} + 1).\] Equation \eqref{eq:u} also allows us to evaluate the integral in \eqref{int>1sqrt2pt2} as \[\int_{0}^\frac{\pi}{4}\left(\sqrt{2}\cos \theta+\sqrt{1-2\sin^2\theta}\right)^2 d\theta= 1 + \frac{3\pi}{4}.\]

Using the same changes of variables (namely $u \mapsto u - u^{-1}$ and $x \mapsto x - x^{-1} = z$),
the second integral in \eqref{intbyprt} can be written as \begin{align*}
&\int_{1}^{\sqrt{2} + 1} \frac{1}{x} \left(\int^x \frac{u(u - u^{-1})}{\sqrt{4 - (u - u^{-1})^2}} du \right) dx\\  =& \frac{1}{4}\int_{0}^2 \left(-\sqrt{16 - z^4} + 2 \arcsin\left(\frac{z^2}{4}\right) - z\sqrt{4-z^2}  + 4\arcsin\left(\frac{z}{2}\right)\right) \frac{dz}{\sqrt{4+z^2}} \\ =& -\frac{1}{4}\int_{0}^2 \sqrt{4-z^2} dz - \frac{1}{4}\int_{0}^2 \frac{z\sqrt{4 - z^2}}{\sqrt{4+z^2 }} dz + \frac{1}{2} \int_{0}^2 \frac{\arcsin \left(\frac{z^2}{4}\right)}{\sqrt{4+z^2}} dz + \int_{0}^2 \frac{\arcsin \left(\frac{z}{2}\right)}{\sqrt{4+z^2}} dz \\ =& -\frac{1}{4}\left[\frac{1}{2} z\sqrt{4-z^2} + 2\arcsin\left(\frac{z}{2}\right)\right]_{0}^2 - \frac{1}{8}\left[\sqrt{16 - z^4} + 4\arcsin\left(\frac{z^2}{2}\right)\right]_{0}^2 \\  &+ \frac{1}{2} \int_{0}^2 \frac{\arcsin \left(\frac{z^2}{4}\right)}{\sqrt{4+z^2}} dz +  \int_{0}^2 \frac{\arcsin \left(\frac{z}{2}\right)}{\sqrt{4+z^2}} dz \\ =& -\frac{\pi}{2} + \frac{1}{2} +  \frac{1}{2}\int_{0}^2 \frac{\arcsin \left(\frac{z^2}{4}\right)}{\sqrt{4+z^2}} dz + \int_{0}^2 \frac{\arcsin \left(\frac{z}{2}\right)}{\sqrt{4+z^2}} dz.  
\end{align*}

Thus, to evaluate \eqref{eq:mmsecond}, it  only remains to evaluate the last two integrals above. The change of variable $v= z/2$ yields \begin{equation}\label{eq:arcsin}\int_{0}^2 \frac{\arcsin \left(\frac{z}{2}\right)}{\sqrt{4+z^2}} dz = \int_{0}^1 \frac{\arcsin v}{\sqrt{1+ v^2}} dv, \quad \int_{0}^2 \frac{\arcsin \left(\frac{z^2}{4}\right)}{\sqrt{4+z^2}} dz = \int_{0}^1 \frac{\arcsin v^2}{\sqrt{1+v^2}} dv.\end{equation}
The first integral in \eqref{eq:arcsin} equals \begin{align*}
\int_{0}^1 \frac{\arcsin v}{\sqrt{1+ v^2}} dv =& \left[\arcsin v \int^v \frac{dw}{\sqrt{1 + w^2}} - \int \frac{1}{\sqrt{1 - v^2}} \left(\int^v \frac{dw}{\sqrt{1+w^2}}\right) dv\right]_{0}^1 \\ =& \left.\arcsin v \log \left(v + \sqrt{1 + v^2}\right)\right|_0^1 - \int_{0}^1 \frac{\log \left(v + \sqrt{1 + v^2}\right)}{\sqrt{1 - v^2}} dv.
\end{align*}
We will use \[\int_{0}^{1} \frac{\log \left(v + \sqrt{1 + v^2}\right)}{\sqrt{1 - v^2}} dv = \ \mbox{Catalan's constant} \ = D(i),\]
which is equation $(26)$ in \cite{Bradley2}, after a suitable change of variables. 
We have that the first integral in \eqref{eq:arcsin} is \[\int_{0}^1 \frac{\arcsin v}{\sqrt{1+ v^2}} dv = \left. \arcsin v \log \left(v + \sqrt{1 + v^2}\right)\right|_{0}^1 - D(i) = \frac{\pi}{2}\log(\sqrt{2} + 1) - D(i).\]

In order to compute the second integral in \eqref{eq:arcsin}, we make the change of variables $u=v^2$ and notice that \begin{align*}
    \int_0^1 \frac{\arcsin (v^2)}{\sqrt{1+v^2}}dv = \frac{1}{2}\int_0^1 \frac{\arcsin (u)}{\sqrt{u(1+u)}}du =& \left.\arcsin(u) \arcsinh(\sqrt{u}) \right|_0^1 -\int_0^1 \frac{ \arcsinh(\sqrt{u})}{\sqrt{1-u^2}}du \\ =& \frac{\pi}{2}\log \left(\sqrt{2} + 1\right) - \int_0^1 \frac{ \arcsinh(\sqrt{u})}{\sqrt{1-u^2}}du.
\end{align*} 
We recall that
\[\arcsinh(\sqrt{u})=\sum_{j=0}^\infty \frac{(-1)^j u^{j+\frac{1}{2}}}{4^j(2j+1)} \binom{2j}{j}.\]
 (See formula 4.6.31 in \cite{AS}.) Thus, 
we have to compute 
\[\int_0^1 \frac{ \arcsinh(\sqrt{u})}{\sqrt{1-u^2}}du=\sum_{j=0}^\infty \frac{(-1)^j}{4^j(2j+1)} \binom{2j}{j}\int_0^1 \frac{u^{j+\frac{1}{2}}}{\sqrt{1-u^2}} du.\] 

The change of variables $v=u^2$ allows us to express the previous integral in terms of the beta function. (See formulas 6.2.1 and 6.2.2 in \cite{AS}.) This gives
\begin{align*}
\int_{0}^1 \frac{u^{j+\frac{1}{2}}}{\sqrt{1-u^2}} du = \frac{1}{2}\int_{0}^1 v^{\frac{2j - 1}{4}}(1-v)^{-\frac{1}{2}} dv =&\frac{\Gamma\left(\frac{2j+3}{4}\right) \Gamma\left(\frac{1}{2}\right)}{2\Gamma\left(\frac{2j+5}{4}\right)}=\frac{2^{j-\frac{1}{2}}\Gamma\left(\frac{2j+3}{4} \right)^2 }{ \Gamma\left(\frac{2j+3}{2}\right)},
\end{align*}
where the last equality follows from  the Lagrange's duplication formula for the Gamma function (Equation 6.1.18 in \cite{AS}) \begin{equation}\label{eq:gamma*}\Gamma(z) \Gamma\left(z+\frac{1}{2}\right)=2^{1-2z} \sqrt{\pi} \Gamma(2z), 
\end{equation}
and the identity $\Gamma\left(\frac{1}{2}\right)=\sqrt{\pi}$. 

Therefore, we have 
\[\int_0^1 \frac{ \arcsinh(\sqrt{u})}{\sqrt{1-u^2}}du=\sum_{j=0}^\infty \frac{(-1)^j}{2^{j+\frac{1}{2}}(2j+1)} \binom{2j}{j}\frac{\Gamma\left(\frac{2j+3}{4} \right)^2 }{ \Gamma\left(\frac{2j+3}{2}\right)}.\]
Using \eqref{eq:gamma*} again, we obtain  
\[\Gamma\left(\frac{2j+3}{2}\right)=\frac{2^{-2j-1} \sqrt{\pi}\Gamma(2j+2)}{\Gamma(j+1)}=\frac{2^{-2j-1} \sqrt{\pi}(2j+1)!}{j!}.\]
Since 
\[\Gamma\left(\frac{2j+3}{4} \right)=\begin{cases}\left(\frac{2j-1}{4} \right)\left(\frac{2j-5}{4} \right)\cdots \frac{3}{4}\left(-\frac{1}{4}\right)\Gamma\left(\frac{-1}{4}\right) & j \mbox{ even},\\\left(\frac{2j-1}{4} \right)\left(\frac{2j-5}{4} \right)\cdots \frac{1}{4}\Gamma\left(\frac{1}{4}\right)& j \mbox{ odd},\end{cases}\]
this finally gives
\begin{align*}\int_0^1 \frac{ \arcsinh(\sqrt{u})}{\sqrt{1-u^2}}du=&\frac{1}{\sqrt{\pi}}\sum_{\ell=0}^\infty \frac{2^{2\ell+\frac{1}{2}}}{(4\ell+1)^2 (2\ell)!}\Gamma\left(\frac{4\ell +3}{4} \right)^2-\frac{1}{\sqrt{\pi}}\sum_{\ell=0}^\infty \frac{2^{2\ell+\frac{3}{2}}}{(4\ell+3)^2(2\ell+1)!} \Gamma\left(\frac{4\ell+5}{4} \right)^2 \\
=&\frac{\Gamma\left(\frac{-1}{4}\right)^2}{\sqrt{\pi}}\sum_{\ell=0}^\infty \frac{\prod_{k=0}^{\ell} (4k-1)^2}{2^{2\ell+\frac{7}{2}}(4\ell+1)^2 (2\ell)!} - \frac{\Gamma\left(\frac{1}{4}\right)^2}{\sqrt{\pi}}\sum_{\ell=0}^\infty \frac{\prod_{k=0}^{\ell} (4k+1)^2}{2^{2\ell+\frac{5}{2}}(4\ell+3)^2(2\ell+1)!} \\ =& \frac{2\Gamma\left(\frac{3}{4}\right)^2}{\sqrt{2\pi}} \hypgeo{4}{3}\left(\frac{1}{4}, \frac{1}{4}, \frac{3}{4}, \frac{3}{4}; \frac{1}{2}, \frac{5}{4}, \frac{5}{4}; 1\right) - \frac{\Gamma\left(\frac{1}{4}\right)^2}{36\sqrt{2\pi}} \hypgeo{4}{3}\left(\frac{3}{4}, \frac{3}{4}, \frac{5}{4}, \frac{5}{4}; \frac{3}{2}, \frac{7}{4}, \frac{7}{4}; 1\right).
\end{align*}

The last equality follows  from comparing the sums with the corresponding expressions for $\hypgeo{4}{3}$ as follows \begin{align*}
\frac{2\Gamma\left(\frac{3}{4}\right)^2}{\sqrt{2\pi}}\hypgeo{4}{3}\left(\frac{1}{4}, \frac{1}{4}, \frac{3}{4}, \frac{3}{4}; \frac{1}{2}, \frac{5}{4}, \frac{5}{4}; 1\right) =& \frac{2\Gamma\left(\frac{3}{4}\right)^2}{\sqrt{2\pi}}\sum_{\ell=0}^\infty \frac{\left(\frac{1}{4}\right)_\ell^2 \left(\frac{3}{4}\right)_\ell^2}{\left(\frac{1}{2}\right)_\ell \left(\frac{5}{4}\right)_{\ell}^2}\frac{1}{\ell!} \\ =& \frac{\Gamma\left(\frac{-1}{4}\right)^2}{2^3\sqrt{2\pi}}\sum_{\ell=0}^\infty \frac{1}{2^{2\ell}(4\ell + 1)^2 (2\ell)!}\prod_{k=0}^{\ell}(4k-1)^2 \\ =& \frac{\Gamma\left(\frac{-1}{4}\right)^2}{\sqrt{\pi}}\sum_{\ell=0}^\infty \frac{\prod_{k=0}^{\ell} (4k-1)^2}{2^{2\ell+\frac{7}{2}}(4\ell+1)^2 (2\ell)!}
,  
\end{align*}
and similarly, \[\frac{\Gamma\left(\frac{1}{4}\right)^2}{36\sqrt{2\pi}} \hypgeo{4}{3}\left(\frac{3}{4}, \frac{3}{4}, \frac{5}{4}, \frac{5}{4}; \frac{3}{2}, \frac{7}{4}, \frac{7}{4}; 1\right) = \frac{\Gamma\left(\frac{1}{4}\right)^2}{\sqrt{\pi}}\sum_{\ell=0}^\infty \frac{\prod_{k=0}^{\ell} (4k+1)^2}{2^{2\ell+\frac{5}{2}}(4\ell+3)^2(2\ell+1)!}.\] 
Therefore we have \begin{align*}
    \int_{1}^{\sqrt{2}+1} \frac{x(x - x^{-1})\log x}{\sqrt{4 - (x - x^{-1})^2}} dx =& \frac{\pi}{2} - \frac{1}{2} + D(i)  + \frac{\Gamma\left(\frac{3}{4}\right)^2}{\sqrt{2\pi}} \hypgeo{4}{3}\left(\frac{1}{4}, \frac{1}{4}, \frac{3}{4}, \frac{3}{4}; \frac{1}{2}, \frac{5}{4}, \frac{5}{4}; 1\right) \\ & - \frac{\Gamma\left(\frac{1}{4}\right)^2}{72\sqrt{2\pi}} \hypgeo{4}{3}\left(\frac{3}{4}, \frac{3}{4}, \frac{5}{4}, \frac{5}{4}; \frac{3}{2}, \frac{7}{4}, \frac{7}{4}; 1\right).
\end{align*}
This concludes the evaluation of \eqref{int>1sqrt2pt1}, and therefore of \eqref{eq:mmsecond}. 
Combining with \eqref{eq:L4}, this concludes the proof.
\end{proof}

\subsection{The areal Mahler measure of $y+\left(\frac{1-x}{1+x}\right)$} 

 In this section we consider the rational function $y+\left(\frac{1-x}{1+x}\right)$. More precisely, we prove Theorem \ref{thm:8}.

\begin{proof}[Proof of Theorem \ref{thm:8}] 
As in previous cases, we have, by definition, \begin{align}\label{eq:moebius}\m_{\mathbb{D}} \left(y+\left(\frac{1-x}{1+x}\right) \right)=&\frac{1}{\pi^2}\int_{\mathbb{D}^2}\log\left|y+\left(\frac{1-x}{1+x}\right) \right|dA(y) dA(x)\nonumber \\
 =& \frac{1}{2\pi} \int_{\mathbb{D}\cap \left\{\left|\frac{1-x}{1+x}\right|< 1 \right\}} \left(\left|\frac{1-x}{1+x}\right|^2-1\right)dA(x) +\frac{1}{\pi} \int_{\mathbb{D}\cap \left\{\left|\frac{1-x}{1+x}\right|\geq 1\right\} } \log \left|\frac{1-x}{1+x}\right|dA(x). 
\end{align}
 
 We consider the first integral above. Note that, for $\theta \in [-\pi, \pi)$ and $x = \rho e^{i\theta},$ \begin{equation}\label{eq:halfplane} \left|\frac{1-x}{1+x}\right| \leq 1 \Leftrightarrow \frac{1 - 2\rho\cos \theta + \rho^2}{1+ 2\rho\cos \theta + \rho^2} \leq 1 \Leftrightarrow \cos \theta \geq 0 \Leftrightarrow \theta \in \left[-\frac{\pi}{2}, \frac{\pi}{2}\right].\end{equation}
Therefore, we have \begin{equation}\label{int<1}
\frac{1}{2\pi}\int_{\D\cap\left\{\left|\frac{1-x}{1+x}\right| \leq 1\right\}} \left(\left|\frac{1-x}{1+x}\right|^2 - 1\right) dA(x) = -\frac{1}{\pi} \int_{0}^1 \rho \int_{-\frac{\pi}{2}}^{\frac{\pi}{2}} \frac{2\rho \cos \theta}{1+ 2\rho \cos \theta + \rho^2} d\theta d\rho.
\end{equation}
The integral with respect to $\theta$ in \eqref{int<1} is evaluated to be \begin{align}
\int_{-\frac{\pi}{2}}^{\frac{\pi}{2}} \frac{2\rho \cos \theta}{1+ 2\rho \cos \theta + \rho^2} d\theta =& 2 \int_{0}^{\frac{\pi}{2}}  \frac{2\rho \cos \theta}{1+ 2\rho \cos \theta + \rho^2} d\theta \nonumber \\ =& 2\left[\int_{0}^{\frac{\pi}{2}} d\theta - (1+\rho^2)\int_{0}^{\frac{\pi}{2}} \frac{1}{1+ 2\rho \cos \theta + \rho^2} d\theta \right] \nonumber \\ =& \pi - 2(1+\rho^2)\int_{0}^{\frac{\pi}{2}} \frac{\sec^2 \left(\frac{\theta}{2}\right)}{(1+\rho)^2 + (1-\rho)^2\tan^2 \left(\frac{\theta}{2}\right)} d\theta. \nonumber 
\end{align}
By applying the change of variables $u= \tan \left(\frac{\theta}{2}\right),$ we find that \begin{align}
\int_{-\frac{\pi}{2}}^{\frac{\pi}{2}} \frac{2\rho \cos \theta}{1+ 2\rho \cos \theta + \rho^2} d\theta =& \pi - 4(1+\rho^2)\int_{0}^{1} \frac{d u}{(1+\rho)^2 + (1-\rho)^2u^2} \nonumber \\ =& \pi - \frac{4(1+\rho^2)}{1-\rho^2}\arctan \left(\frac{1-\rho}{1+\rho}\right). \label{int<1pt1}  
\end{align}
Incorporating \eqref{int<1pt1} in \eqref{int<1}, we obtain \begin{align}
\int_{0}^1 \rho \int_{-\frac{\pi}{2}}^{\frac{\pi}{2}} \frac{2\rho \cos \theta}{1+ 2\rho \cos \theta + \rho^2} d\theta d\rho =& \int_{0}^1 \rho \left[ \pi - \frac{4(1+\rho^2)}{1-\rho^2}\arctan \left(\frac{1-\rho}{1+\rho}\right)\right] d\rho \nonumber \\ =& \frac{\pi}{2} - \int_{0}^1 \frac{4\rho(1+\rho^2)}{1-\rho^2}\arctan\left(\frac{1-\rho}{1+\rho}\right) d\rho \nonumber \\ =& \frac{\pi}{2} - \left[\left. \arctan\left(\frac{1-\rho}{1+\rho}\right) \int^\rho \frac{4r(1+r^2)}{1-r^2} dr\right|^1_0 + \int_0^1 \frac{1}{1 + \rho^2}  \int^\rho \frac{4r(1+r^2)}{1-r^2} dr d\rho \right]\label{int<1pt1*}.
\end{align} 

Applying the change of variables $v= 1 - \rho^2,$ we have \[\int \frac{4\rho(1+\rho^2)}{1-\rho^2} d\rho =-2 \int \frac{2-v}{v}dv = -4\log v + 2v +C = -4\log(1-\rho^2) + 2(1-\rho^2)+C.\]

Then, from \eqref{int<1pt1*}, we derive that \begin{align}
\int_{0}^1 \rho \int_{-\frac{\pi}{2}}^{\frac{\pi}{2}} \frac{2\rho \cos \theta}{1+ 2\rho \cos \theta + \rho^2} d\theta d\rho =& \frac{\pi}{2} - \left[-\frac{\pi}{2} - \int_{0}^1 \frac{1}{1 + \rho^2} \left(4\log(1-\rho^2) - 2(1-\rho^2)\right) d\rho   \right] \nonumber \\ =& \pi + 4\int_{0}^1 \frac{\log(1-\rho^2)}{1+\rho^2}  d\rho - 2\int_{0}^1 \frac{1-\rho^2}{1+\rho^2}d\rho \nonumber \\ =& \pi \log 2 + 2 + 4\int_{0}^1 \frac{\log\left(\frac{1-\rho^2}{2}\right)}{1+\rho^2}  d\rho \nonumber \\ =& 2 + \pi \log 2 - 4D(i), \label{int<1pt1**}
\end{align} 
where the last equality from the integral representation of the Catalan's constant $D(i)$ (equation $(19)$ in \cite{Bradley2}).

By incorporating the result of \eqref{int<1pt1**}  into \eqref{int<1}, we obtain \begin{align}\label{int:D}
\frac{1}{2\pi}\int_{\D\cap\left\{\left|\frac{1-x}{1+x}\right| \leq 1\right\}} \left(\left|\frac{1-x}{1+x}\right|^2 - 1\right) dA(x)\\
=& \frac{4D(i)}{\pi}- \log 2 - \frac{2}{\pi}.
\end{align}

It remains to evaluate the second integral in \eqref{eq:moebius}.
Recall from \eqref{eq:halfplane} that we have \[\left|\frac{1-x}{1+x}\right| \geq 1  \Leftrightarrow \cos \theta \leq 0 \Leftrightarrow \theta \in \left[-\pi, -\frac{\pi}{2}\right) \cup \left(\frac{\pi}{2}, \pi\right),\] where $x = \rho e^{i\theta}$ and $-\pi \leq \theta < \pi.$

Therefore, \eqref{eq:moebius} can be written as \begin{align}
\int_{\D\cap \left\{\left|\frac{1-x}{1+x}\right| \geq 1 \right\}} \log \left|\frac{1-x}{1+x}\right| dA(x) =& \frac{1}{2}\int_{0}^1 \rho \left[\int_{-\pi}^{-\frac{\pi}{2}} \log \left(\frac{1 - 2\rho \cos \theta + \rho^2}{1 + 2\rho \cos \theta + \rho^2}\right) d\theta \right. \nonumber \\ &\left. + \int_{\frac{\pi}{2}}^{\pi} \log \left(\frac{1 - 2\rho \cos \theta + \rho^2}{1 + 2\rho \cos \theta + \rho^2}\right) d\theta \right] d\rho \nonumber \\ =& \int_{0}^1 \rho \left[\int_{\frac{\pi}{2}}^{\pi} \log \left(\frac{1 - 2\rho \cos \theta + \rho^2}{1 + 2\rho \cos \theta + \rho^2}\right) d\theta \right] d\rho. \label{int>1}
\end{align}

We develop the power series of $\log \left(\frac{1 - 2\rho \cos \theta + \rho^2}{1 + 2\rho \cos \theta + \rho^2}\right)$ to get
\begin{align*}
\log \left(\frac{1 - 2\rho \cos \theta + \rho^2}{1 + 2\rho \cos \theta + \rho^2}\right)=& \log \left(1-\frac{2\rho \cos \theta}{1+\rho^2}\right) -\log \left(1+\frac{2\rho \cos \theta}{1+\rho^2}\right)\\
=& \sum_{k=1}^\infty \frac{(-1)^{k}-1}{k} \left(\frac{2\rho \cos \theta }{1+\rho^2}\right)^k\\
=&-2\sum_{j=0}^\infty \frac{1}{2j+1} \left(\frac{2\rho \cos \theta }{1+\rho^2}\right)^{2j+1}. 
\end{align*}  

Now, for $n \geq 2,$ a repetitive use of the fact \[\int \cos^n \theta d\theta = \frac{1}{n}\cos^{n-1} \theta \sin \theta + \frac{n-1}{n}\int \cos^{n-2} \theta d\theta\] yields \begin{align*}
\int \cos^n \theta d\theta =& \frac{1}{n}\cos^{n-1} \theta \sin \theta + \frac{1}{n}\sum_{k=1}^{\left \lfloor \frac{n}{2} \right\rfloor - 1} \left(\prod_{\ell =1}^k \frac{n-2\ell  + 1}{n-2\ell }\right) \cos^{n-2k-1} \theta \sin \theta \\  &+ \frac{(n-1)(n-3)\cdots \left(n - 2\left \lfloor \frac{n}{2} \right\rfloor + 1\right)}{n(n-2)\cdots \left(n - 2\left \lfloor \frac{n}{2} \right\rfloor + 2\right)} \int \cos^{n - 2\left \lfloor \frac{n}{2} \right\rfloor} \theta d\theta,
\end{align*}
where $\lfloor x \rfloor$ denotes the largest integer less than or equal to $x.$

Therefore, when $n = 2j + 1,$ we have, for $j =0$ and $j \geq 1,$ \[\int_{\frac{\pi}{2}}^{\pi} \cos \theta d\theta = -1, \quad  \int_{\frac{\pi}{2}}^{\pi} \cos^{2j + 1} \theta d\theta = \frac{(2j)(2j-2)\cdots 2}{(2j+1)(2j-1)\cdots 3} \int_{\frac{\pi}{2}}^{\pi} \cos \theta d\theta = - \frac{4^j (j!)^2}{(2j+1)!},\] respectively. This implies that \begin{align}
\int_{\frac{\pi}{2}}^{\pi} \log \left(\frac{1 - 2\rho \cos \theta + \rho^2}{1 + 2\rho \cos \theta + \rho^2}\right) d\theta  
=& 2\left(\frac{2\rho}{1+\rho^2} + \sum_{j=1}^{\infty} \frac{1}{2j+1}\frac{4^j (j!)^2}{(2j+1)!^2}\left(\frac{2\rho}{1+\rho^2}\right)^{2j+1}\right). \label{int>1fullint}
\end{align}
Therefore, in order to compute the integral over $\rho$ in \eqref{int>1}, we need to first consider the individual integrals \begin{equation}\label{eq:rhointegrals}\int_{0}^1 \left(\frac{2\rho}{1+\rho^2}\right)^{2j+1} \rho d \rho, \quad \quad \mbox{for all} \ j \geq 0.\end{equation}

When $j=0,$ we have \begin{equation*}
\int_{0}^1 \frac{2\rho^2}{1+\rho^2} d \rho = 2\int_{0}^1 \left(1 - \frac{1}{1 + \rho^2}\right) d\rho = 2 \left[\rho - \arctan\rho \right]_{0}^1 = 2 - \frac{\pi}{2}.
\end{equation*}

For $j \geq 1,$ using integration by parts (where $\frac{\rho}{(1+\rho^2)^{2j+1}}$ is integrated and the rest is differentiated), the integrals in \eqref{eq:rhointegrals} give \begin{align}
\int_{0}^1 \left(\frac{2\rho}{1+\rho^2}\right)^{2j+1} \rho d \rho =& 4^j \left(\left. \frac{-\rho^{2j+1}}{2j\left(1+\rho^2\right)^{2j}}\right|^1_0 + \frac{2j+1}{2j} \int_0^1 \frac{\rho^{2j}}{\left(1+\rho^2\right)^{2j}} d\rho\right) \nonumber \\ =& 4^j \left[\frac{-\rho^{2j+1}}{2j\left(1+\rho^2\right)^{2j}} + \frac{2j+1}{2(2j)} \left(-\frac{\rho^{2j-1}}{(2j-1)\left(1+\rho^2\right)^{2j-1}} + \int \frac{\rho^{2j-2}}{\left(1+\rho^2\right)^{2j-1}} d\rho \right)\right]_{0}^1 \nonumber \\ =& -\frac{2}{2j-1} + \frac{4^j(2j+1)}{4j} \int_0^1 \frac{\rho^{2j-2}}{\left(1+\rho^2\right)^{2j-1}} d\rho. \label{int>1ind}
\end{align} 
The change of variables $u=\rho^2$ yields  \[\int_{0}^1 \left(\frac{2\rho}{1+\rho^2}\right)^{2j+1} \rho d \rho = -\frac{2}{2j-1} + \frac{4^{j-1}(2j+1)}{2j} \int_0^1 \frac{u^{j-\frac{3}{2}}}{\left(1+u\right)^{2j-1}} du.\]
Making the change of variables $v=\frac{u}{1+u}$, we have 
\begin{align*}
 \int_0^1 \frac{u^{j-\frac{3}{2}}}{\left(1+u\right)^{2j-1}} du=& \int_0^\frac{1}{2} v^{j-\frac{3}{2}} (1-v)^{j-\frac{3}{2}}dv
\end{align*}
By the change of variables $w=1-v$, we have that 
\[\int_0^\frac{1}{2} v^{j-\frac{3}{2}} (1-v)^{j-\frac{3}{2}}dv=\int_\frac{1}{2}^1 w^{j-\frac{3}{2}} (1-w)^{j-\frac{3}{2}}dw.\]
Therefore, we obtain the beta function
\begin{align}
 \int_0^1 \frac{u^{j-\frac{3}{2}}}{\left(1+u\right)^{2j-1}} du=& \frac{1}{2}\int_0^1 v^{j-\frac{3}{2}} (1-v)^{j-\frac{3}{2}}dv=\frac{\Gamma\left(j-\frac{1}{2}\right)^2}{2\Gamma\left(2j-1\right)}=\frac{4^{3-2j}\pi \Gamma\left(2j-2\right)^2}{2\Gamma(j-1)^2\Gamma\left(2j-1\right)}\nonumber\\
 =&\frac{\pi}{2^{4j-3}}\binom{2j-2}{j-1}. \label{int>1indeval}
\end{align}
 Incorporating \eqref{int>1indeval} into \eqref{int>1ind}, we obtain 
\begin{equation}\label{int>1ind*}
\int_{0}^1 \left(\frac{2\rho}{1+\rho^2}\right)^{2j+1} \rho d \rho =  -\frac{2}{2j-1} + \frac{(2j+1)\pi }{4^jj}\binom{2j-2}{j-1}.
\end{equation}

Next, combining \eqref{int>1} and \eqref{int>1fullint} along with \eqref{int>1ind*}, we derive that \begin{align*}
\int_{\D \cap \left\{\left|\frac{1-x}{1+x}\right| \geq 1 \right\}} \log \left|\frac{1-x}{1+x}\right| dA(x) =& 2\int_{0}^1 \left(\frac{2\rho}{1+\rho^2} + \sum_{j=1}^{\infty} \frac{1}{2j+1}\frac{4^j (j!)^2}{(2j+1)!}\left(\frac{2\rho}{1+\rho^2}\right)^{2j+1}\right) \rho d\rho \\ =& 2\left(2 - \frac{\pi}{2} + \sum_{j=1}^{\infty} \frac{1}{2j+1}\frac{4^j (j!)^2}{(2j+1)!}\left(-\frac{2}{2j-1} + \frac{(2j+1)\pi}{4^jj}\binom{2j-2}{j-1}\right)\right) \\ =& 4 - \pi - 2\sum_{j=1}^{\infty} \frac{2}{4j^2 - 1}\frac{4^j (j!)^2}{(2j+1)!} + \pi \sum_{j=1}^{\infty} \frac{1}{j^2(2j+1)}\frac{(j!)^2}{(2j-1)!}\binom{2j-2}{j-1}.
\end{align*}

Simplifying the above sums individually, we obtain \begin{equation*}
\pi \sum_{j=1}^{\infty} \frac{1}{j^2(2j+1)}\frac{(j!)^2}{(2j-1)!}\binom{2j-2}{j-1} = \pi \sum_{j=1}^{\infty} \frac{1}{4j^2 - 1} = \frac{\pi}{2} \sum_{j=1}^{\infty} \left(\frac{1}{2j - 1} - \frac{1}{2j+1}\right) = \frac{\pi}{2},
\end{equation*}
and
\begin{align*}
\sum_{j=1}^{\infty} \frac{2}{4j^2 - 1}\frac{4^j (j!)^2}{(2j+1)!} =& \sum_{j=1}^{\infty} \frac{1}{2j-1}\frac{4^j (j!)^2}{(2j+1)!} - \sum_{j=1}^{\infty} \frac{1}{2j+1}\frac{4^j (j!)^2}{(2j+1)!} \\ =& \frac{1}{2}\sum_{j=1}^{\infty} \left(\frac{1}{2j-1} - \frac{1}{2j+1}\right) \frac{4^j (j!)^2}{(2j)!} - \sum_{j=1}^{\infty} \frac{1}{2j+1}\frac{4^j (j!)^2}{(2j+1)!} \\ =& 1 + \frac{1}{2} \sum_{j=2}^{\infty} \frac{2j}{(2j-1)^2}\frac{4^{j-1} ((j-1)!)^2}{(2(j-1))!} - \frac{1}{2}\sum_{j=2}^{\infty} \frac{1}{2j-1}\frac{4^{j-1} ((j-1)!)^2}{(2(j-1))!}  \\ &-  \sum_{j=1}^{\infty} \frac{1}{2j+1}\frac{4^j (j!)^2}{(2j+1)!}  \\ =& 1 + \frac{1}{2}\sum_{k=1}^{\infty} \frac{1}{2k+1}\frac{4^k (k!)^2}{(2k+1)!} - \sum_{j=1}^{\infty} \frac{1}{2j+1}\frac{4^j (j!)^2}{(2j+1)!}  \\ =& 1 - \frac{1}{2}(2D(i)-1) \\ =& \frac{3}{2} - D(i),
\end{align*}
where the evaluation of the sum follows from  \cite[Theorem 2]{Bradley}. (See also equation (61) in \cite{Bradley2}.) Therefore, the integral over the domain $\D \cap \left\{\left|\frac{1-x}{1+x}\right| \geq 1 \right\}$ yields \begin{equation*}
\frac{1}{\pi}\int_{\D \cap\left\{\left|\frac{1-x}{1+x}\right| \geq 1 \right\}} \log \left|\frac{1-x}{1+x}\right| dA(x) = \frac{1}{\pi}\left[4 - \pi -2\left(\frac{3}{2} - D(i)\right) + \frac{\pi}{2}\right] = \frac{1}{\pi} - \frac{1}{2} + \frac{2}{\pi}D(i).
\end{equation*}  
By combining this with \eqref{int:D} and \eqref{eq:L4}, the result follows.

 \end{proof}

\section{The areal versions of generalized, higher,  and zeta Mahler measures} \label{sec:generalized}
In this section we consider areal versions of some variants of the Mahler measure, namely generalized Mahler measure, higher Mahler measure, and zeta Mahler measures.

\subsection{Generalized areal Mahler measure} 
Generalized Mahler measures were introduced by Gon and Oyanagi \cite{GonOyanagi} who studied their basic properties, computed some examples, and related them to multiple sine functions and special values of Dirichlet $L$-functions. They were also studied in \cite{IL,L08}. 

Given nonzero rational functions $P_1, \dots, P_r \in \C(x_1, \dots, x_n)^\times$, the generalized (logarithmic) Mahler measure of $P_1, \dots, P_r$ is defined by
\[
\m_{\max}(P_1,\dots, P_r) = \frac{1}{(2\pi i)^n} \int_{\TT^n} \max\{\log|P_1|, \dots, \log|P_r|\} \frac{d x_1}{x_1} \cdots  \frac{d x_n}{x_n}.\]

It is natural to consider the generalized areal Mahler measure of $P_1, \dots, P_r$ to be defined by
\[
\m_{\mathbb{D},\max}(P_1,\dots, P_r) = \frac{1}{\pi^n} \int_{\mathbb{D}^n} \max\{\log|P_1|, \dots, \log|P_r|\} d A(x_1) \dots dA(x_n).\]
We have the following result. 
\begin{prop}
 We have
 \[\m_{\mathbb{D},\mathrm{max}}(x_1,\dots,x_n)=-\frac{1}{2n}.\]
\end{prop}
We remark that in the classical case, we have $\m_{\mathrm{max}}(x_1,\dots,x_n)=0$.
\begin{proof}
As usual, we proceed to apply the definition together with the change of variables $x_j=\rho_j e^{i\theta_j}$. This gives
\begin{align*}\m_{\mathbb{D},\mathrm{max}}(x_1,\dots,x_n)=&\frac{1}{\pi^n}\int_{\mathbb{D}^n} \max \{\log|x_1|,\dots,\log|x_n|\} dA(x_1)\dots dA(x_n)\\
 =& 2^n \int_0^1\cdots   \int_0^1\max \{\log\rho_1,\dots,\log\rho_n\} \rho_1\cdots \rho_n d\rho_1 \dots d\rho_n.
\end{align*}
Notice that the above integral can be written as a sum of $n!$ integrals, where each term corresponds to a certain order  of the variables $\rho_j$. The advantage of considering the ordered variables lies in the fact that the maximum is then easy to describe. Thus, the above becomes  
 \begin{align*}
 & 2^n n! \int_{0\leq \rho_1\leq \dots \leq \rho_n \leq 1} \log \rho_n \rho_1\cdots \rho_n d\rho_1 \dots d\rho_n\\ 
 =& 2^n n! \int_{0\leq \rho_2\leq \dots \leq \rho_n \leq 1} \log \rho_n \frac{\rho_2^3}{2}\cdots \rho_n d\rho_2 \dots d\rho_n\\ 
 =& 2^n n! \int_{0\leq \rho_3\leq \dots \leq \rho_n \leq 1} \log \rho_n \frac{\rho_3^5}{2\cdot 4}\cdots \rho_n d\rho_3 \dots d\rho_n\\ 
 =& \cdots \\
 =& 2^n n! \int_{0\leq \rho_n \leq 1} \log \rho_n \frac{\rho_n^{2n-1}}{2\cdots (2n-2)}d\rho_n\\ 
 =& 2n \left. \frac{\rho_n^{2n}}{(2n)^2} (2n \log \rho_n-1)\right|^1_0\\
 =&-\frac{1}{2n}.
 \end{align*}
\end{proof}

\subsection{Multiple and higher areal Mahler measures} 

The multiple  Mahler measure was defined in \cite{KLO} for nonzero rational functions $P_1, \dots, P_r \in \C(x_1, \dots, x_n)^\times$ by
\[\m(P_1,\dots, P_r) :=\frac{1}{(2\pi i)^n}\int_{\TT^n}\log|P_1(x_1, \dots, x_n)|\cdots
\log|P_r(x_1, \dots, x_n)| \frac{dx_1}{x_1}\cdots \frac{dx_n}{x_n}.\]
For the particular case in which $P_1=\dots=P_r=P$, the multiple Mahler measure is called {\em higher Mahler measure} and is given by 
\begin{equation}\label{eq:higher}
\m_r(P) :=\frac{1}{(2\pi i)^n}\int_{\TT^n}\log^r|P(x_1, \dots, x_n)|\frac{dx_1}{x_1}\cdots \frac{dx_n}{x_n}.
\end{equation}
The multiple Mahler measure and the higher Mahler measure were considered by various authors who computed specific formulas and proved various limiting properties \cite{KLO, Sasaki1, Sasaki2, BS, BBSW, IL, Biswas, BiswasMonico, LL}. 
It is natural to consider the areal versions of these constructions, given by 
\[\m_{\mathbb{D},h_1,\dots,h_r} (P_1,\dots, P_r) :=\frac{1}{\pi^n}\int_{\mathbb{D}^n}\log^{h_1}|P_1(x_1, \dots, x_n)|\cdots
\log^{h_r}|P_r(x_1, \dots, x_n)| dA(x_1)\dots dA(x_n).\]
\begin{prop}
 We have 
 \[\m_{\mathbb{D},h_1,\dots,h_n}(x_1,\dots,x_n)=\frac{(-1)^{h_1+\cdots+h_n}h_1!\cdots h_n!}{2^{h_1+\cdots+h_n+n}}.\]
\end{prop}
We remark that in the classical case, we have $\m_{h_1,\dots,h_n}(x_1,\dots,x_n)=0$.
\begin{proof}
First we recall equation \eqref{eq:xjlogk}, which in particular gives 
\[\int_0^1 x \log^k x dx =\frac{(-1)^kk!}{2^{k+1}}.\]
Proceeding as usual by setting $x_j=\rho_j e^{i\theta_j}$, we have 
\begin{align*}\m_{\mathbb{D},h_1,\dots,h_n}(x_1,\dots,x_n)=&\frac{1}{\pi^n}\int_{\mathbb{D}^n} \log^{h_1}|x_1|\cdots \log^{h_n}|x_n| dA(x_1)\dots dA(x_n)\\
 =& 2^n \int_0^1 \cdots \int_0^1  \log^{h_1}\rho_1 \cdots \log^{h_n}\rho_n \rho_1\cdots \rho_n d\rho_1 \dots d\rho_n\\ 
 =& \frac{(-1)^{h_1+\cdots+h_n}h_1!\cdots h_n!}{2^{h_1+\cdots+h_n+n}}.
 \end{align*}
\end{proof}
We consider a more elaborate case. The following formula is proven in \cite{SasakiIJNT, LL}:
\[\m_h\left(\frac{1-x}{1+x}\right) = \begin{cases}{\displaystyle\frac{|E_h|}{2^h}\pi^h} & h \text{ even},\\ \\
0 & h \text{ odd},
\end{cases}
\]
where $E_n$ denotes the $n$th Euler number. 

We have the following  areal version of the above result.
\begin{thm}\label{thm:multipleCayley} For $h \in \Z_{>0}$ even, we have, 
\begin{align*}
 \m_{\D, h}\left(\frac{1-x}{1+x}\right) =&-2(h-1) B_{h}(\pi i)^{h-1}+
\frac{E_{h}(\pi i)^{h}}{2^{h}}-
\frac{ E_{h-2}(\pi i)^{h-2}h(h-1)}{2^{h-2}}\log 2
\\&-\frac{4 h!}{2^h}\sum_{\substack{m=2\\} }^{h-1} (1-2^{1-m})\zeta(m) 
\frac{E_{h-m-1}(\pi i)^{h-m-1}}{(h-m-1)!}
\\&
-2(\pi i)^{h-1}\sum_{\substack{m=0} }^{h} \binom{h}{m} (1-2^{1-m})  (1-2^{1-h+m})B_{m}B_{h-m},   
\end{align*}
where $B_n$ and $E_n$ denote the $n$th Bernoulli number and the $n$th Euler number respectively, and the first sum for $h=2$ should be interpreted as equal to zero. 

For $h$ odd,  we have \[\m_{\D, h}\left(\frac{1-x}{1+x}\right)=0.\]

\end{thm}

Before proceeding to the proof of this result, we need the following lemma.

\begin{lem}\label{lem:reductions}
 Let $h \in \Z_{>2}$  $|\alpha|, |\beta|\leq 1$. We have 
 \begin{align}\label{eq:length2}
\sum_{b> 1}\frac{\beta^b}{b^{h-1}} \sum_{a=1}^{b-1} \frac{\alpha^a}{a} 
=& \Li_{1,h-1}(\alpha,\beta),
\end{align}
and for $\alpha \not =1$,
\begin{align}\label{eq:length1}
\sum_{b> 1}\frac{\beta^b}{b^{h+1}} \sum_{a=1}^{b-1} a\alpha^a 
=&\frac{1}{(\alpha-1)^2}
(\alpha\Li_h(\alpha\beta)-\alpha\Li_{h+1}(\alpha\beta)-\Li_h(\alpha\beta)+\alpha\Li_{h+1}(\beta)).
\end{align}
\end{lem}
\begin{proof}
Identity \eqref{eq:length2} follows directly from the definition of multiple polylogarithms. For identity \eqref{eq:length1} we have     
\begin{align*}
\sum_{b> 1}\frac{\beta^b}{b^{h+1}} \sum_{a=1}^{b-1} a\alpha^a =&\sum_{b\geq 1 }\frac{\beta^b}{b^{h+1}}\frac{\alpha((b-1)\alpha^b-b\alpha^{b-1}+1)}{(\alpha-1)^2} \\
=&\frac{1}{(\alpha-1)^2}
(\alpha\Li_h(\alpha\beta)-\alpha\Li_{h+1}(\alpha\beta)-\Li_h(\alpha\beta)+\alpha\Li_{h+1}(\beta)).
\end{align*}
\end{proof}

\begin{proof}[Proof of Theorem \ref{thm:multipleCayley}]
By definition, we have 
\begin{align*}
 \m_{\D, h}\left(\frac{1-x}{1+x}\right) = & \frac{1}{\pi}\int_{\D} \log^h\left|\frac{1-x}{1+x}\right|dA(x).  \end{align*}
We remark that the change of variables $z=\frac{1-x}{1+x}$ takes the unit disk to the right half plane $\mathbb{H}_+$ defined by $\re(z)\geq 0$. The areal measure of the unit disk is $dA(x) = dx_1 dx_2,$ and under the map $f: x \mapsto \frac{1-x}{1+x} = z,$ by definition, the areal measure of $\HH_+$ is \[dA(z) = \left|x_{1,z_1} x_{2, z_2} - x_{1, z_2} x_{2, z_1}\right|dz_1 dz_2, \quad \mbox{where} \quad \ \frac{\partial z_j}{\partial x_k} = z_{j, x_k} \ \mbox{for} \ j=1, 2 \ \mbox{and} \ k=1,2.\] Since the map $f$ is conformal on $\D$, it satisfies the Cauchy--Riemann relations, and therefore \[\left|x_{1,z_1} x_{2, z_2} - x_{1, z_2} x_{2, z_1}\right|dz_1dz_2 = \left|\frac{d x}{dz}\right|^2 dz_1 dz_2 =\left|\frac{d}{dz}\left(\frac{1-z}{1+z}\right)\right|^2 dz_1 dz_2 = \frac{4dz_1dz_2}{|z+1|^4}.\]
In sum, we have that \begin{equation}\label{eq:conformal}\int_{\D} \log^h \left|\frac{1-x}{1+x}\right| dA(x) = 4\int_{\HH_{+}} \log^h |z| \frac{dz_1dz_2}{|z+1|^4}.
\end{equation}

We further consider the change to polar coordinates  $z = \rho e^{i\theta}$ with $\rho \geq 0$,
$-\frac{\pi}{2}\leq \theta \leq \frac{\pi}{2}$. The integral in \eqref{eq:conformal} can be written as \begin{align*}
\int_{\HH_{+}} \log^h |z| \frac{dz_1dz_2}{|z+1|^4} =& \int_{\re(z)\geq 0} \log^h |z| \frac{dz_1dz_2}{|z+1|^4}\\ =& \int_{0}^{\infty} \rho \log^h \rho \int_{-\frac{\pi}{2}}^{\frac{\pi}{2}} \frac{d\theta}{\left|\rho e^{i\theta}+1\right|^4} d\rho  \\ =& \int_{0}^{1} \rho \log^h \rho \int_{-\frac{\pi}{2}}^{\frac{\pi}{2}} \frac{d\theta}{\left|\rho e^{i\theta}+1\right|^4} d\rho + (-1)^h \int_{0}^1 \rho \log^h \rho \int_{-\frac{\pi}{2}}^{\frac{\pi}{2}} \frac{d\theta}{\left|\rho e^{i\theta}+1\right|^4} d\rho,
\end{align*} 
where the last equality is obtained by separating the integral over $\rho$ into $0\leq \rho \leq 1$ and $1\leq \rho$, and then applying the change of variables $\rho \mapsto \rho^{-1}$ in the $1\leq \rho$ term. The above derivation implies that \[\int_{\HH_{+}} \log^h |z| \frac{dz_1dz_2}{|z+1|^4} = \begin{cases} 
{\displaystyle 2
\int_{0}^{1} \rho \log^h \rho \int_{-\frac{\pi}{2}}^{\frac{\pi}{2}} \frac{d\theta}{\left|\rho e^{i\theta}+1\right|^4} d\rho} & h \ \mbox{is even},\\ \\
0 & h \ \mbox{is odd}.
\end{cases}\] 
For what follows we will assume that $h$ is even. First  we will also assume that $h\not =2$, as this will guarantee that certain series converge. The case $h=2$ will be treated later.

Evaluating the individual terms in the above formula leads to\begin{align}\int_{0}^{1} \rho \log^h \rho \int_{-\frac{\pi}{2}}^{\frac{\pi}{2}} \frac{d\theta}{\left|\rho e^{i\theta}+1\right|^4} d\rho =& \int_{0}^{1} \rho \log^h \rho \int_{-\frac{\pi}{2}}^{\frac{\pi}{2}} \frac{d\theta}{\left(\rho e^{i\theta}+1\right)^2 \left(\rho e^{-i\theta}+1\right)^2} d\rho \nonumber \\ =& \int_{0}^{1} \rho \log^h \rho \int_{-\frac{\pi}{2}}^{\frac{\pi}{2}} \left(\sum_{k \geq 0}(k+1)(-\rho)^k e^{ik\theta}\right)\left(\sum_{\ell \geq 0}(\ell+1)(-\rho)^\ell e^{-i\ell
\theta}\right) d\theta d\rho \nonumber \\ =& \pi \sum_{k \geq 0} \int_{0}^1 (k+1)^2 \rho^{2k+1} \log^h \rho d\rho \label{inthighamm1}\\ &+ \frac{2}{i}\sum_{k > \ell \geq 0} \int_{0}^1 \frac{(k+1)(\ell+1)}{k-\ell} (-1)^{k+\ell} \left(i^{k-\ell} - (-i)^{k-\ell}\right) \rho^{k+\ell+1} \log^h \rho d\rho.  \label{inthighamm}
\end{align}
From \eqref{eq:xjlogk} we have \[\int_{0}^1 x^j \log^k x dx = \frac{(-1)^k k!}{(j+1)^{k+1}}.\] Since $h$ is even, we can calculate the integral in \eqref{inthighamm1} and obtain, for $h>2$, \[\sum_{k \geq 0} \int_{0}^1 (k+1)^2 \rho^{2k+1} \log^h \rho d\rho = \sum_{k \geq 0} \frac{h! (k+1)^2}{(2k+2)^{h+1}} =  \frac{h!}{2^{h+1}} \sum_{k \geq 0} \frac{1}{(k+1)^{h-1}} = \frac{h!}{2^{h+1}} \zeta (h-1).\] 

For the  integral in \eqref{inthighamm}, first notice that, since
 $k+\ell$ and $k-\ell$ have the same parity, it suffices to consider\[\sum_{k > \ell \geq 0} \int_{0}^1 \frac{(k+1)(\ell+1)}{k-\ell} I_{k-\ell} \rho^{k+\ell+1} \log^h \rho d\rho, \quad \mbox{where} \quad \ I_{j}= (-i)^j - i^j.\]  
Setting $k^*=k+1$, $\ell^*=\ell+1$ first, and then $a=k^*-\ell^*$, $b=k^*+\ell^*$, we 
find that \begin{align*}
\sum_{k > \ell \geq 0} \int_{0}^1 \frac{(k+1)(\ell+1)}{k-\ell} I_{k-\ell} \rho^{k+\ell+1} \log^h \rho d\rho =& h! \sum_{k > \ell \geq 0} \frac{(k+1)(\ell+1)}{(k-\ell)(k+\ell+2)^{h+1}} I_{k-\ell} \\ =& h! \sum_{k^* > \ell^* \geq 1} \frac{k^* \ell^*}{(k^* - \ell^*)(k^* + \ell^*)^{h+1}} I_{k^* - \ell^*} \\ =& \frac{h!}{4} \left[\sum_{k^* > \ell^* \geq 1} \frac{I_{k^* - \ell^*}}{(k^* - \ell^*)(k^* + \ell^*)^{h-1}} - \sum_{k^* > \ell^* \geq 1} \frac{(k^* - \ell^*) I_{k^* - \ell^*}}{(k^* + \ell^*)^{h+1}} \right] \\ 
=& \frac{h!}{4} \left[\sum_{\substack{b > a\geq 1\\a \equiv b \mod{2}}} \frac{I_a}{ab^{h-1}} - \sum_{\substack{b > a\geq 1\\a\equiv b \mod{2}}} \frac{a I_a}{b^{h+1}} \right] \\
=& \frac{h!}{8} \left[\sum_{\substack{b > a\geq 1}} \frac{(1+(-1)^{a+b})I_a}{ab^{h-1}} - \sum_{\substack{b > a\geq 1}} \frac{a (1+(-1)^{a+b})I_a}{b^{h+1}} \right]\\
 =& \frac{h!}{8} \left[\sum_{\substack{b > a\geq 1}} \frac{(-i)^a-i^a+i^a(-1)^b-(-i)^a(-1)^b}{ab^{h-1}} \right.\\&\left.- \sum_{\substack{b > a\geq 1}} \frac{a((-i)^a-i^a+i^a(-1)^b-(-i)^a(-1)^b)}{b^{h+1}} \right].
\end{align*}

Applying Lemma \ref{lem:reductions} gives
\begin{align*}
&\sum_{k > \ell \geq 0} \int_{0}^1 \frac{(k+1)(\ell+1)}{k-\ell} I_{k-\ell} \rho^{k+\ell+1} \log^h \rho d\rho\\
=&\frac{h!}{8} \left[\sum_{\substack{b > a\geq 1}} \frac{(-i)^a-i^a+i^a(-1)^b-(-i)^a(-1)^b}{ab^{h-1}}- \sum_{\substack{b > a\geq 1}} \frac{a((-i)^a-i^a+i^a(-1)^b-(-i)^a(-1)^b)}{b^{h+1}} \right]\\
=&\frac{h!}{8} [\Li_{1,h-1}(-i,1)- \Li_{1,h-1}(i,1)+\Li_{1,h-1}(i,-1)-\Li_{1,h-1}(-i,-1)\\
&-\frac{1}{(-i-1)^2}\left( -i\Li_h(-i)+i\Li_{h+1}(-i)-\Li_h(-i)-i\Li_{h+1}(1)\right)\\
&+\frac{1}{(i-1)^2}\left(  i\Li_h(i)-i\Li_{h+1}(i)-\Li_h(i)+i\Li_{h+1}(1)\right)\\
&-\frac{1}{(i-1)^2}\left( 
i\Li_h(-i)-i\Li_{h+1}(-i)-\Li_h(-i)+i\Li_{h+1}(-1)\right)\\
&+\frac{1}{(-i-1)^2}\left( 
-i\Li_h(i)+i\Li_{h+1}(i)-\Li_h(i)-i\Li_{h+1}(-1)\right)].
\end{align*}
The length 2 polylogarithms above can be written in terms of length 1 polylogarithms by means of Theorem \ref{thm:LL} as follows (recall that $h>2$ is even),  
\begin{align*}
&\Li_{1,h-1}(-i,1)- \Li_{1,h-1}(i,1)+\Li_{1,h-1}(i,-1)-\Li_{1,h-1}(-i,-1)\\
=&2{\re}_h(\Li_{1,h-1}(-i,1))+2{\re}_h(\Li_{1,h-1}(i,-1))\\
=&3\Li_h(i)+\Li_h(-i)+(h-1)(\Li_h(1)+\Li_h(-1))\\
&-(\Li_{h-1}(1)+\Li_{h-1}(-1))(- \Li_{1}(-i)+\Li_{1}(i))\\&+\sum_{m=1}^{h-1} (\Li_m(i)+\Li_m(-i))(- \Li_{h-m}(-i)+(-1)^{m-1}\Li_{h-m}(i)).
\end{align*}

For the length 1 polylogarithms in the expression of $\sum_{k > \ell \geq 0} \int_{0}^1 \frac{(k+1)(\ell+1)}{k-\ell} I_{k-\ell} \rho^{k+\ell+1} \log^h \rho d\rho$, we have 
\begin{align*}
&-\frac{1}{(-i-1)^2}\left( -i\Li_h(-i)+i\Li_{h+1}(-i)-\Li_h(-i)-i\Li_{h+1}(1)\right)\\
&+\frac{1}{(i-1)^2}\left(  i\Li_h(i)-i\Li_{h+1}(i)-\Li_h(i)+i\Li_{h+1}(1)\right)\\
&-\frac{1}{(i-1)^2}\left( 
i\Li_h(-i)-i\Li_{h+1}(-i)-\Li_h(-i)+i\Li_{h+1}(-1)\right)\\
&+\frac{1}{(-i-1)^2}\left( 
-i\Li_h(i)+i\Li_{h+1}(i)-\Li_h(i)-i\Li_{h+1}(-1)\right)\\
=&\frac{i}{2}\left( -i\Li_h(-i)+i\Li_{h+1}(-i)-i\Li_{h+1}(1)\right)+\frac{i}{2}\left(  i\Li_h(i)-i\Li_{h+1}(i)+i\Li_{h+1}(1)\right)\\
&-\frac{i}{2}\left( 
i\Li_h(-i)-i\Li_{h+1}(-i)+i\Li_{h+1}(-1)\right)-\frac{i}{2}\left( 
-i\Li_h(i)+i\Li_{h+1}(i)-i\Li_{h+1}(-1)\right)\\
=&-(\Li_h(i)-\Li_h(-i))+(\Li_{h+1}(i)-\Li_{h+1}(-i)).
\end{align*}

Putting everything together, we have
\begin{align*}
&\frac{8}{h!}\sum_{k > \ell \geq 0} \int_{0}^1 \frac{(k+1)(\ell+1)}{k-\ell} I_{k-\ell} \rho^{k+\ell+1} \log^h \rho d\rho\\
=&2(\Li_h(i)+\Li_h(-i))+(h-1)(\Li_h(1)+\Li_h(-1))\\
&-(\Li_{h-1}(1)+\Li_{h-1}(-1))(- \Li_{1}(-i)+\Li_{1}(i))\\&+\sum_{m=1}^{h-1} (\Li_m(i)+\Li_m(-i))(- \Li_{h-m}(-i)+(-1)^{m-1}\Li_{h-m}(i))\\&+(\Li_{h+1}(i)-\Li_{h+1}(-i))\\
=&2^{2-h}\Li_h(-1)+(h-1)2^{1-h} \zeta(h)-2^{2-h}\zeta(h-1)\frac{\pi i}{2}+
2i L(h+1,\chi_{-4})\\&+\sum_{\substack{m=1\\m\, \text{odd}} }^{h-1}2^{1-m} \Li_m(-1)2i L(h-m,\chi_{-4})-\sum_{\substack{m=1\\m\, \text{even}} }^{h-1}2^{1-m} \Li_m(-1)  2^{1-h+m}\Li_{h-m}(-1) \\
=&2^{2-h}(2^{1-h}-1)\zeta(h)+(h-1)2^{1-h} \zeta(h)-i\pi 2^{1-h}\zeta(h-1)+
2i L(h+1,\chi_{-4})-2i \log(2) L(h-1,\chi_{-4})\\&+i\sum_{\substack{m=2\\m\, \text{odd}} }^{h-1}2^{2-m} (2^{1-m}-1)\zeta(m) L(h-m,\chi_{-4})-\sum_{\substack{m=2\\m\, \text{even}} }^{h-1}2^{2-h} (2^{1-m}-1)  (2^{1-h+m}-1)\zeta(m)\zeta(h-m).
\end{align*}
We can further simplify by applying the facts that 
\[\zeta(2n) =\frac{(-1)^{n+1}B_{2n}(2\pi)^{2n}}{2(2n)!}\quad \mbox{ and }\quad L(2n+1,\chi_{-4})= \frac{(-1)^n E_{2n}\pi^{2n+1}}{2^{2n+2}(2n)!}.\]
The previous identity then equals
\begin{align*}
&2(1-2^{1-h})\frac{i^hB_{h}\pi^{h}}{h!}-(h-1) \frac{i^hB_{h}\pi^{h}}{h!}-i\pi 2^{1-h}\zeta(h-1)+
\frac{i^{h+1} E_{h}\pi^{h+1}}{2^{h+1}h!}+\log(2) 
\frac{i^{h+1} E_{h-2}\pi^{h-1}}{2^{h-1}(h-2)!}
\\&+\sum_{\substack{m=2\\m\, \text{odd}} }^{h-1}2 (2^{1-m}-1)\zeta(m) 
\frac{i^{h-m} E_{h-m-1}\pi^{h-m}}{2^{h}(h-m-1)!}
-\frac{i^h \pi^h}{h!} \sum_{\substack{m=2\\m\, \text{even}} }^{h-1} \binom{h}{m} (2^{1-m}-1)  (2^{1-h+m}-1)B_{m}B_{h-m}.
\end{align*}   

Finally we get, for $h>2$ even, 
\begin{align}\label{eq:crazyheigher}
 &\frac{1}{\pi}\int_{\D} \log^h \left|\frac{1-x}{1+x}\right| dA(x) \nonumber\\=&\frac{h!}{2^{h-2}}\zeta(h-1)+ \frac{2h!}{\pi i}\left [2(1-2^{1-h})\frac{i^hB_{h}\pi^{h}}{h!}-(h-1) \frac{i^hB_{h}\pi^{h}}{h!}-i\pi 2^{1-h}\zeta(h-1)\right. \nonumber
\\&\left.+
\frac{i^{h+1} E_{h}\pi^{h+1}}{2^{h+1}h!}+\log(2) 
\frac{i^{h+1} E_{h-2}\pi^{h-1}}{2^{h-1}(h-2)!}+\sum_{\substack{m=2\\m\, \text{odd}} }^{h-1}2 (2^{1-m}-1)\zeta(m) 
\frac{i^{h-m} E_{h-m-1}\pi^{h-m}}{2^{h}(h-m-1)!}\right.\nonumber
\\&\left.
-\frac{i^h \pi^h}{h!} \sum_{\substack{m=2\\m\, \text{even}} }^{h-1} \binom{h}{m} (2^{1-m}-1)  (2^{1-h+m}-1)B_{m}B_{h-m} 
 \right] \nonumber \\
 =&-2(h-1) B_{h}(\pi i)^{h-1}+
\frac{E_{h}(\pi i)^{h}}{2^{h}} -\log(2) 
\frac{ E_{h-2}(\pi i)^{h-2}h(h-1)}{2^{h-2}}\nonumber\\&-\frac{4h!}{2^h}\sum_{\substack{m=2\\} }^{h-1}(1-2^{1-m})\zeta(m) 
\frac{E_{h-m-1}(\pi i)^{h-m-1}}{(h-m-1)!}\nonumber \\
&-2(\pi i)^{h-1} \sum_{\substack{m=0} }^{h} \binom{h}{m} (1-2^{1-m})  (1-2^{1-h+m})B_{m}B_{h-m},
\end{align}
where we have used that $B_n=E_n=0$ when $n$ is odd (with the exception of $B_1$).

 Note that above computation fails to converge when $h=2.$ Therefore, we need to evaluate the $h = 2$ case in a different way. Since $\log^2 \left|\frac{1-x}{1+x}\right| = \log^2 \left|1-x \right| - 2\log\left|1-x\right|\log\left|1+x\right| + \log^2 \left|1+x\right|,$ and $\m_{\D, 2}\left(1-x\right) = \m_{\D, 2}\left(1+x\right),$ we have \[\m_{\D, 2}\left(\frac{1-x}{1+x}\right) = 2\m_{\D, 2}\left(1+x\right) - \frac{2}{\pi}\int_{\D}\log\left|1-x\right|\log\left|1+x\right| dx.\] It only remains to compute the second integral. Following the method in the proof of \cite[Theorem 7]{KLO}, we derive \begin{align*}
 \int_{\D} \log\left|1-x\right|\log\left|1+x\right| dx =& \int_{0}^1 \rho \left[\int_{0}^{2\pi} \re \left(\log\left(1-\rho e^{i\theta}\right)\right) \re \left(\log\left(1+\rho e^{i\theta}\right)\right)d\theta \right] d\rho \\ =& \int_{0}^1 \rho \left[ \int_{0}^{2\pi} \left(-\sum_{k \geq 1} \frac{\rho^k}{k} \cos (k\theta)\right)\left(-\sum_{\ell \geq 1} \frac{(-1)^\ell \rho^\ell}{\ell} \cos (\ell\theta)\right) d\theta \right] d\rho \\ =& \int_{0}^1 \sum_{k, \ell \geq 1} \frac{(-1)^\ell \rho^{k+\ell+1}}{k\ell} \left[\int_{0}^{2\pi} \cos (k\theta) \cos (\ell\theta) d\theta \right] d\rho.
 \end{align*}
 On the other hand, we have \[\int_{0}^{2\pi} \cos (k\theta) \cos (\ell\theta) d\theta = \begin{cases} \pi & \quad \quad \mbox{if} \ k=\ell, \\ 0 & \quad \quad \mbox{if} \ k \neq \ell. \end{cases}\] This implies that \begin{align*}
 \int_{\D} \log\left|1-x\right|\log\left|1+x\right| dx =& \int_{0}^1 \sum_{k, \ell \geq 1} \frac{(-1)^\ell \rho^{k+\ell+1}}{k\ell} \left[\int_{0}^{2\pi} \cos (k\theta) \cos (\ell\theta) d\theta \right] d\rho \\ =& \pi \sum_{k \geq 1} \frac{(-1)^k}{k^2} \int_{0}^1 \rho^{2k+1} d\rho \\ =& \frac{\pi}{2} \sum_{k \geq 1} \frac{(-1)^k}{k^2(k+1)} \\ =& \frac{\pi}{2} \sum_{k \geq 1} (-1)^k \left[\frac{1}{k^2} - \frac{1}{k} + \frac{1}{k+1}\right] \\ =& \frac{\pi}{2}\left[\Li_2(-1) - 2\Li_1(-1) - 1\right].
 \end{align*}

 Therefore, 
 \begin{align*}\m_{\D, 2}\left(\frac{1-x}{1+x}\right) =& 2\m_{\D, 2}\left(1+x\right) - \Li_2(-1) + 2\Li_1(-1) + 1\\
 =&\frac{\pi^2}{6}-1+\frac{\pi^2}{12} - 2\log 2 + 1\\
 =& \frac{\pi^2}{4}-2\log 2.
 \end{align*}
Finally we remark that this value is also obtained by replacing $h=2$ in \eqref{eq:crazyheigher}.  
\end{proof}

\subsection{Areal zeta Mahler measure}

The zeta Mahler measure was defined by Akatsuka \cite{Akatsuka} for a nonzero rational function $P\in \C(x_1,\dots,x_n)^\times$ by 
\[ 
Z(s,P) 
:= \frac{1}{(2\pi i)^n} \int_{\TT^n} 
\left|P\left(x_1, \dots, x_n\right)\right|^s 
\frac{dx_1}{x_1}\cdots \frac{dx_n}{x_n},
\]
where $s$ is a complex variable in a neighborhood of $0$.

The zeta Mahler measure was considered in \cite{KLO, Biswas, BiswasMurty, SasakiIJNT, Ringeling}.
Its Taylor expansion is the generating series of the higher Mahler measure given in  \eqref{eq:higher}.
\[Z(s,P) = \sum_{k=0}^\infty \frac{\m_k(P) s^k}{k!}.\]
Akatsuka  computed 
the zeta Mahler measure
$Z(s,x-c)$ for a constant $c$. His result in \cite[Theorem 2]{Akatsuka} implies the following formula for $\rho<1$:
\begin{equation}\label{eq:Akatsuka}\frac{1}{2\pi}\int_{-\pi}^\pi  |\rho e^{i\theta}+1|^s d\theta=
\hypgeo{2}{1}
\left( -\frac{s}{2},-\frac{s}{2};1;\rho^2 \right).
\end{equation}

We may consider the areal zeta Mahler measure defined by 
\[Z_\mathbb{D}(s,P) 
:= \frac{1}{\pi ^n} \int_{\mathbb{D}^n} 
\left|P\left(x_1, \dots, x_n\right)\right|^s 
dA(x_1)\dots dA(x_n).\]

We follow some arguments from \cite[Theorem 14]{KLO} to prove the next result. 
\begin{thm}
 We have  
 \[ Z_\mathbb{D}(s,x+1)=\exp \left(\sum_{j=2}^\infty \frac{(-1)^{j}}{j} (1-2^{1-j})(\zeta(j)-1)s^j\right).\]
\end{thm}
\begin{proof}
By definition and the usual change of variables, we have  
 \begin{align*}
 Z_\mathbb{D}(s,x+1)=&\frac{1}{\pi}\int_{\mathbb{D}} |x+1|^s dA(x)
 = \frac{1}{\pi}\int_0^1 \int_{-\pi}^\pi  |\rho e^{i\theta}+1|^s \rho d\theta d\rho.
\end{align*}
By applying Akatsuka's result \eqref{eq:Akatsuka},  we then have to compute 
 \begin{align*}
  Z_\mathbb{D}(s,x+1)=&2  \int_0^1 \hypgeo{2}{1}\left( -\frac{s}{2},-\frac{s}{2};1;\rho^2 \right) \rho d\rho 
  =2  \int_0^1 \sum_{n=0}^\infty \frac{(-s/2)_n^2}{(n!)^2}\rho^{2n+1} d\rho\\
   =& \sum_{n=0}^\infty \binom{s/2}{n}^2\frac{1}{(n+1)}
   =\frac{1}{s/2+1} \sum_{n=0}^\infty
         \binom{s/2}{n} \binom{s/2+1}{s/2-n}=\frac{1}{s/2+1} \binom{s+1}{s/2}\\
   =&\frac{s+1}{(s/2+1)^2}\frac{\Gamma(s+1)}{\Gamma(s/2+1)^2}. 
 \end{align*}
We now apply the Weierstrass product of the Gamma function 
 \[\Gamma(s+1)^{-1}=e^{\gamma s}\prod_{k=1}^\infty \left(1+\frac{s}{k}\right)e^{-s/k},\]
 to obtain
  \begin{align*}
  Z_\mathbb{D}(s,x+1)=&\prod_{k=2}^\infty \frac{\left(1+\frac{s}{2k}\right)^2}{1+\frac{s}{k}}\\
  =& \exp \left(\sum_{k=2}^\infty \left( 2\log \left(1+\frac{s}{2k}\right) -\log\left (1+\frac{s}{k}\right)\right)\right)\\
  =& \exp \left(\sum_{j=1}^\infty \frac{(-1)^{j-1}}{j}\sum_{k=2}^\infty \left( \frac{2}{(2k)^j} -\frac{1}{k^j} \right)s^j\right)\\
   =& \exp \left(\sum_{j=1}^\infty \frac{(-1)^{j}}{j} (1-2^{1-j})\sum_{k=2}^\infty  \frac{s^j}{k^j}\right)\\
    =& \exp \left(\sum_{j=2}^\infty \frac{(-1)^{j}}{j} (1-2^{1-j})(\zeta(j)-1)s^j\right).
  \end{align*}
\end{proof} 
As the areal zeta Mahler measure is the generating function of the areal higher Mahler measure, it can be used to compute this last one by taking  successive derivatives and evaluating at $s=0$. For example, the first few values are given by
 \begin{align*}
 \m_{\mathbb{D},1}(x+1)=&0,\\ 
  \m_{\mathbb{D},2}(x+1)=&\frac{\zeta(2)-1}{2},\\ 
  \m_{\mathbb{D},3}(x+1)=&-\frac{3(\zeta(3)-1)}{2},\\
  \m_{\mathbb{D},4}(x+1)=&\frac{3(7\zeta(4)+\zeta(2)^2-2\zeta(2)-6)}{4}=\frac{3(19\zeta(4)-4\zeta(2)-12)}{8},\\
   \m_{\mathbb{D},5}(x+1)=&-\frac{15(3\zeta(5)+\zeta(3)\zeta(2)-\zeta(3)-\zeta(2)-2)}{2}.
 \end{align*}

 \section{Conclusion}

 We have obtained  formulas for the areal Mahler and its generalizations for various rational functions. In most cases the results are connected to evaluations of polylogarithms leading to special values of functions with arithmetic significance such as the Riemann zeta function, and Dirichlet $L$-functions. In this sense, the results are analogous to what is obtained in the case of the standard Mahler measure. 
 
 However, there is a crucial difference between the areal case and the standard case.  In the standard case, there is a way to assign a weight to the terms in the formula, that typically results in formulas of homogeneous weight 1.  For this, we recall that the length one $n$-th polyogarithm has weight  $n$. The logarithm is associated to $\Li_1(z)$ and it has therefore weight 1. The constant $\pi$ arises as  an evaluation of the logarithm and therefore has weight 1. Finally, we assign weight one to the Mahler measure itself. Taking the weight multiplicatively, we have, for example, in Smyth's formula \eqref{eq:Smyth} that $\m(1+x+y)$ has weight 1, while $L(\chi_{-3},2)$ has weight 2 (as it is the evaluation of a dilogarithm), while $\pi$ has weight 1, giving a total weight of 1 on the right-hand side. In contrast, the terms on the right-hand side of the areal Mahler measure \eqref{eq:arealsmyth} do not have a homogeneous weight. While the term $\frac{3\sqrt{3}}{4\pi}L(\chi_{-3},2)$ has weight 1, the term $\frac{1}{6}$ has weight 0 and the term $-\frac{11\sqrt{3}}{16\pi}$ has weight $-1$. This suggests that if there is a connection between $\m_\D$ and the regulator, it will be more difficult to describe than in the standard case. It would be nevertheless interesting to explore the possibility of such connection. 

 None of the formulas considered in this article correspond to rational functions whose Mahler measure is related to  special values of other $L$-functions, such as $L$-functions attached to elliptic curves. For example,  the following formula was conjectured by Deninger  \cite{Deninger} and Boyd \cite{Bo98}  and proven by Rogers and Zudilin \cite{RZ14}:
\[\m\left(x+\frac{1}{x}+y+\frac{1}{y}+1\right)=L'(E_{15a8},0),\]
where $L(E_{15a8},s)$ denotes the $L$-function associated to the elliptic curve $15a8$. 
A natural question and direction of future research would be to evaluate the areal version of the above formula.

\bibliographystyle{amsalpha}

\bibliography{Bibliography}

\providecommand{\bysame}{\leavevmode\hbox to3em{\hrulefill}\thinspace}
\providecommand{\MR}{\relax\ifhmode\unskip\space\fi MR }
% \MRhref is called by the amsart/book/proc definition of \MR.
\providecommand{\MRhref}[2]{%
  \href{http://www.ams.org/mathscinet-getitem?mr=#1}{#2}
}
\providecommand{\href}[2]{#2}
\begin{thebibliography}{BBSW12}

\bibitem[Aka09]{Akatsuka}
Hirotaka Akatsuka, \emph{Zeta {M}ahler measures}, J. Number Theory \textbf{129}
  (2009), no.~11, 2713--2734. \MR{2549527}

\bibitem[AS64]{AS}
Milton Abramowitz and Irene~A. Stegun, \emph{Handbook of mathematical functions
  with formulas, graphs, and mathematical tables}, National Bureau of Standards
  Applied Mathematics Series, No. 55, U. S. Government Printing Office,
  Washington, D.C., 1964, For sale by the Superintendent of Documents.
  \MR{0167642}

\bibitem[BBSW12]{BBSW}
David Borwein, Jonathan~M. Borwein, Armin Straub, and James Wan, \emph{Log-sine
  evaluations of {M}ahler measures, {II}}, Integers \textbf{12} (2012), no.~6,
  1179--1212. \MR{3011556}

\bibitem[BG69]{BernsteinGelfand}
I.~N. Bern\v{s}te\u{\i}n and S.~I. Gel'fand, \emph{Meromorphy of the function
  {$P^{\lambda }$}}, Funkcional. Anal. i Prilo\v{z}en. \textbf{3} (1969),
  no.~1, 84--85. \MR{0247457}

\bibitem[Bis14]{Biswas}
Arunabha Biswas, \emph{Asymptotic nature of higher {M}ahler measure}, Acta
  Arith. \textbf{166} (2014), no.~1, 15--22. \MR{3273494}

\bibitem[BM14]{BiswasMonico}
Arunabha Biswas and Chris Monico, \emph{Limiting value of higher {M}ahler
  measure}, J. Number Theory \textbf{143} (2014), 357--362. \MR{3227353}

\bibitem[BM18]{BiswasMurty}
Arunabha Biswas and M.~Ram Murty, \emph{The zeta {M}ahler measure of
  {$(z^n-1)/(z-1)$}}, Hardy-Ramanujan J. \textbf{41} (2018), 77--84.
  \MR{3935500}

\bibitem[Boy81]{B1}
David~W. Boyd, \emph{Speculations concerning the range of {M}ahler's measure},
  Canad. Math. Bull. \textbf{24} (1981), no.~4, 453--469. \MR{644535}

\bibitem[Boy92]{Boyd-sharp}
\bysame, \emph{Two sharp inequalities for the norm of a factor of a
  polynomial}, Mathematika \textbf{39} (1992), no.~2, 341--349. \MR{1203290}

\bibitem[Boy98]{Bo98}
\bysame, \emph{Mahler's measure and special values of {$L$}-functions},
  Experiment. Math. \textbf{7} (1998), no.~1, 37--82. \MR{1618282}

\bibitem[Bra99]{Bradley}
David~M. Bradley, \emph{A class of series acceleration formulae for {C}atalan's
  constant}, Ramanujan J. \textbf{3} (1999), no.~2, 159--173. \MR{1703281}

\bibitem[Bra01]{Bradley2}
\bysame, \emph{Representations of {C}atalan's constant}, 2001.

\bibitem[BS12]{BS}
Jonathan~M. Borwein and Armin Straub, \emph{Log-sine evaluations of {M}ahler
  measures}, J. Aust. Math. Soc. \textbf{92} (2012), no.~1, 15--36.
  \MR{2945674}

\bibitem[BZ20]{BrunaultZudilin}
Fran\c{c}ois Brunault and Wadim Zudilin, \emph{Many variations of {M}ahler
  measures---a lasting symphony}, Australian Mathematical Society Lecture
  Series, vol.~28, Cambridge University Press, Cambridge, 2020. \MR{4382435}

\bibitem[CS12]{ChoiSamuels}
Kwok-Kwong~Stephen Choi and Charles~L. Samuels, \emph{Two inequalities on the
  areal {M}ahler measure}, Illinois J. Math. \textbf{56} (2012), no.~3,
  825--834. \MR{3161353}

\bibitem[Den97]{Deninger}
Christopher Deninger, \emph{Deligne periods of mixed motives, {$K$}-theory and
  the entropy of certain {${\bf Z}^n$}-actions}, J. Amer. Math. Soc.
  \textbf{10} (1997), no.~2, 259--281. \MR{1415320}

\bibitem[Fla15]{Flammang}
V.~Flammang, \emph{The {M}ahler measure and its areal analog for totally
  positive algebraic integers}, J. Number Theory \textbf{151} (2015), 211--222.
  \MR{3314210}

\bibitem[GO04]{GonOyanagi}
Yasuro Gon and Hideo Oyanagi, \emph{Generalized {M}ahler measures and multiple
  sine functions}, Internat. J. Math. \textbf{15} (2004), no.~5, 425--442.
  \MR{2072087}

\bibitem[GR07]{Gradshteyn-Ryzhik}
I.~S. Gradshteyn and I.~M. Ryzhik, \emph{Table of integrals, series, and
  products}, seventh ed., Elsevier/Academic Press, Amsterdam, 2007, Translated
  from the Russian, Translation edited and with a preface by Alan Jeffrey and
  Daniel Zwillinger, With one CD-ROM (Windows, Macintosh and UNIX).
  \MR{2360010}

\bibitem[IL13]{IL}
Zahraa Issa and Matilde Lal\'in, \emph{A generalization of a theorem of {B}oyd
  and {L}awton}, Canad. Math. Bull. \textbf{56} (2013), no.~4, 759--768.
  \MR{3121685}

\bibitem[KLO08]{KLO}
N.~Kurokawa, M.~Lal\'in, and H.~Ochiai, \emph{Higher {M}ahler measures and zeta
  functions}, Acta Arith. \textbf{135} (2008), no.~3, 269--297. \MR{2457199}

\bibitem[Knu97]{Knuth}
Donald~E. Knuth, \emph{The art of computer programming. {V}ol. 1},
  Addison-Wesley, Reading, MA, 1997, Fundamental algorithms, Third edition [of
  MR0286317]. \MR{3077152}

\bibitem[Lal08]{L08}
Matilde~N. Lal\'in, \emph{Mahler measures and computations with regulators}, J.
  Number Theory \textbf{128} (2008), no.~5, 1231--1271. \MR{2406490}

\bibitem[LL16]{LL}
Matilde~N. Lal\'in and Jean-S\'ebastien Lechasseur, \emph{Higher {M}ahler
  measure of an {$n$}-variable family}, Acta Arith. \textbf{174} (2016), no.~1,
  1--30. \MR{3517530}

\bibitem[LL18]{LalinLechasseur2}
Matilde~N. Lal\'{\i}n and Jean-S\'{e}bastien Lechasseur, \emph{A reduction
  formula for length-two polylogarithms and some applications}, Rev. Un. Mat.
  Argentina \textbf{59} (2018), no.~2, 285--309. \MR{3900276}

\bibitem[Mai00]{CM}
Vincent Maillot, \emph{G\'eom\'etrie d'{A}rakelov des vari\'et\'es toriques et
  fibr\'es en droites int\'egrables}, no.~80, 2000. \MR{1775582}

\bibitem[Nak12]{Nakamura}
Takashi Nakamura, \emph{A simple proof of the functional relation for the
  {L}erch type {T}ornheim double zeta function}, Tokyo J. Math. \textbf{35}
  (2012), no.~2, 333--337. \MR{3058710}

\bibitem[Pan17]{Panzer}
Erik Panzer, \emph{The parity theorem for multiple polylogarithms}, J. Number
  Theory \textbf{172} (2017), 93--113. \MR{3573145}

\bibitem[Pri08]{Pritsker}
Igor~E. Pritsker, \emph{An areal analog of {M}ahler's measure}, Illinois J.
  Math. \textbf{52} (2008), no.~2, 347--363. \MR{2524641}

\bibitem[Rin22]{Ringeling}
Berend Ringeling, \emph{Special zeta {M}ahler functions}, Res. Number Theory
  \textbf{8} (2022), no.~2, Paper No. 29, 27. \MR{4412570}

\bibitem[RV99]{RV}
F.~Rodriguez-Villegas, \emph{Modular {M}ahler measures. {I}}, Topics in number
  theory ({U}niversity {P}ark, {PA}, 1997), Math. Appl., vol. 467, Kluwer Acad.
  Publ., Dordrecht, 1999, pp.~17--48. \MR{1691309}

\bibitem[RZ14]{RZ14}
Mathew Rogers and Wadim Zudilin, \emph{On the {M}ahler measure of
  {$1+X+1/X+Y+1/Y$}}, Int. Math. Res. Not. IMRN (2014), no.~9, 2305--2326.
  \MR{3207368}

\bibitem[Sas10]{Sasaki1}
Yoshitaka Sasaki, \emph{On multiple higher {M}ahler measures and multiple {$L$}
  values}, Acta Arith. \textbf{144} (2010), no.~2, 159--165. \MR{2669717}

\bibitem[Sas12]{Sasaki2}
\bysame, \emph{On multiple higher {M}ahler measures and {W}itten zeta values
  associated with semisimple {L}ie algebras}, Commun. Number Theory Phys.
  \textbf{6} (2012), no.~4, 771--784. \MR{3068407}

\bibitem[Sas15]{SasakiIJNT}
\bysame, \emph{Zeta {M}ahler measures, multiple zeta values and {$L$}-values},
  Int. J. Number Theory \textbf{11} (2015), no.~7, 2239--2246. \MR{3440457}

\bibitem[Smy81]{S1}
C.~J. Smyth, \emph{On measures of polynomials in several variables}, Bull.
  Austral. Math. Soc. \textbf{23} (1981), no.~1, 49--63. \MR{615132}

\end{thebibliography}

\end{document}